\documentclass[10pt]{amsart}
\usepackage{mathrsfs,latexsym,amsfonts,amssymb}
\usepackage{hyperref}
\newtheorem{theorem}{Theorem}[section]
\newtheorem{proposition}[theorem]{Proposition}
\newtheorem{lemma}[theorem]{Lemma}
\newtheorem{corollary}[theorem]{Corollary}
\newtheorem{question}[theorem]{Question}

\newtheorem{example}[theorem]{Example}
\theoremstyle{definition}
\newtheorem{definition}[theorem]{Definition}
\theoremstyle{remark}
\newtheorem{remark}[theorem]{Remark}

\begin{document}
\title[Some properties of Pre-topological groups]
{Some properties of Pre-topological groups}

\author{Fucai Lin*}
\address{Fucai Lin: 1. School of mathematics and statistics,
Minnan Normal University, Zhangzhou 363000, P. R. China; 2. Fujian Key Laboratory of Granular Computing and Application, Minnan Normal University, Zhangzhou 363000, China}
\email{linfucai@mnnu.edu.cn; linfucai2008@aliyun.com}

\author{Ting Wu}
\address{Ting Wu: 1. School of mathematics and statistics,
Minnan Normal University, Zhangzhou 363000, P. R. China}
\email{473207095@qq.com}

\author{Yufan Xie}
\address{Yufan Xie: 1. School of mathematics and statistics,
Minnan Normal University, Zhangzhou 363000, P. R. China}
\email{981553469@qq.com}

\author{Meng Bao}
\address{Meng Bao: College of Mathematics, Sichuan University, Chengdu 610064, China}
\email{mengbao95213@163.com}

\thanks{The first author is supported by the Key Program of the Natural Science Foundation of Fujian Province (No: 2020J02043) and the NSFC (No. 11571158).}
\thanks{* Corresponding author}


\date{\today}
\begin{abstract}
In this paper, we pose the concepts of pre-topological groups and some generalizations of pre-topological groups. First, we systematically investigate some basic properties of pre-topological groups; in particular, we prove that each $T_{0}$ pre-topological group is regular and every almost topological group is completely regular which extends A.A. Markov's theorem to the class of almost topological groups. Moreover, it is shown that an almost topological group is $\tau$-narrow if and only if it can be embedded as a subgroup of a pre-topological product of almost topological groups of weight less than or equal to $\tau$. Finally, the cardinal invariant, the precompactness and the resolvability are investigated in the class of pre-topological groups.
\end{abstract}

\maketitle

\section{Introduction}
Since from 20th century, many topologists and algebraists have contributed to topological algebra. A.D. Alexandroff, N. Bourbaki, E. van Kampen, A.A. Markov and L.S. Pontryagin were among the first contributors to the theory of topological groups. In 2008, the book "Topological Groups and Related Structures" was published, which points out the direction for the research of topological groups. It mainly studied the generic questions in topological algebra is how the relationship between topological properties depend on the underlying algebraic structure. As we all known, a topological group, that is, a group $G$ is endowed with a topology such that the binary operation $G\times G\rightarrow G$ is jointly continuous and the inverse mapping $In: G\rightarrow G$, i.e. $x\rightarrow x^{-1}$, is also continuous. The the properties of topological groups have been widely used in the study of topology, analysis and category, see \cite{A2002,AR2005,BT, LF31,LF32,LF33,LF34,LF35,LF36,S2000,Y2002}. For more details about topological groups, the reader see \cite{AT}. Moreover, the topologies on some kinds of weaker algebra structures than groups are posed and investigated, such as (strongly) topological gyrogroups and rectifiable spaces. Hence it is natural to consider extending some well known results of topological groups to these weaker structures, see \cite{BL,BL2,BL3,BZX,BX,LF,LF1,LF2,LF3}.

In 1999, Doignon and Falmagne 1999 introduced the theory of knowledge spaces (KST) which is regarded as a mathematical framework for the assessment of knowledge and advices for further learning \cite{doignon1985spaces,falmagne2011learning}. KST makes a dynamic evaluation process; of course, the accurate dynamic evaluation is based on individuals' responses to items and the quasi-order on domain $Q$\cite{doignon1985spaces}. In 2009, Danilov discussed the knowledge spaces based on the topological point of view. Indeed, the notion of a knowledge space is a generalization of topological spaces \cite{2009Danilov}, that is, a {\it generalized topology} on a set $Z$ is a subfamily $\mathscr{T}$ of $2^{Z}$ such that $\mathscr{T}$ is closed under arbitrary unions.
Cs\'{a}sz\'{a}r (2002) in \cite{Generalized2002} introduced the notions of generalized topological spaces and then investigated some properties of generalized topological spaces, see \cite{Generalized2002,Generalized2004, Generalized2008,Generalized2009}. Further, J. Li first discuss the pre-topology (that is, the subbase for the topology) with the applications in the theory of  rough sets, see \cite{lijinjin2004,lijinjin2006,lijinjin2007}, and then D. Liu in \cite{liudejin2011,liudejin2013} discuss some properties of pre-topology. Recently, Lin, Cao and Li \cite{LCL} systematically investigated some properties of pre-topology. Based on the theory of pre-topology, we can consider the pre-topology on groups and pose the concept of pre-topological groups, then we can systematically investigate some basic properties of pre-topological groups.

This paper is organized as follows. In Section 2, we introduce the
necessary notation and terminology which are used in
the paper. In Sections 3 and 4, some basic properties of pre-topological groups and relationships among (quasi, para, almost) pre-topological groups are investigated. In Section 5, we mainly study quotient spaces of pre-topological groups and give the three isomorphisms of quotient spaces. In Section 6, we extend A.A. Markov's theorem to almost topological group and show that every almost topological group is completely regular. In Section 7, some cardinal invariants of pre-topological groups are studied. In particular, the well-known Guran's Theorem is extended, that is, an almost topological group is $\tau$-narrow if and only if it can be embedded as a subgroup of a pre-topological product of almost topological groups of weight less than or equal to $\tau$. In Section 8, some properties about precompactness and resolvability are investigated.

\smallskip
\section{Introduction and preliminaries}
Denote the sets of real number, rational number, positive integers, the closed unit interval and all non-negative integers by $\mathbb{R}$, $\mathbb{Q}$, $\mathbb{N}$, $I$ and $\omega$, respectively. Readers may refer
\cite{Eng, LCL} for terminology and notations not
explicitly given here.

\begin{definition}\cite{Generalized2002,lijinjin2004, 2009Danilov}
A {\it pre-topology} on a set $Z$ is a subfamily $\mathscr{T}$ of $2^{Z}$ such that $\bigcup\mathscr{T}=Z$ and $\bigcup\mathscr{T}^{\prime}\in\mathscr{T}$  for any $\mathscr{T}^{\prime}\subseteq\mathscr{T}$. Each element of $\mathscr{T}$ is called an {\it open set} of the pre-topology.
\end{definition}

The name of pre-topology is given in \cite{LCL}. Next, we list some definitions of pre-topological spaces which are introduced in \cite{LCL}; of course, these definitions have important roles in our discussion of pre-topological groups.

\begin{definition}\cite{LCL}
Let $(G, \tau)$ be a pre-topological space and $\mathcal{B}\subseteq \tau$. If for each $U\in\tau$ there exists a subfamily  $\mathcal{B}^{\prime}$ of $\mathcal{B}$ such that $U=\bigcup\mathcal{B}^{\prime}$, then we say that $\mathcal{B}$ is a {\it pre-base} of $(G, \tau)$.
\end{definition}

Let $Z$ be a pre-topological space. For each $z\in Z$, the infimum of the set $$\{|\mathcal{B}(z)|: \mathcal{B}(z)\ \mbox{is a pre-base at}\ z\}$$
is called the {\it character of a point $z$ } \cite{LCL} in $Z$, which is denoted by $\chi(z, Z)$. The supremum of the set $\{\chi(z, Z): z\in Z\}$ is called the {\it character of a pre-topological space $(Z, \tau)$}, which is denoted by $\chi(Z)$. Each set of cardinal numbers being well-ordered by $<$. Then the infimum of the set $$\{|\mathcal{B}|: \mathcal{B}\ \mbox{is a pre-base for}\ Z\}$$ is said to be the {\it weight} of $Z$ \cite{LCL} and is denoted by $w(Z)$.

\begin{definition}\cite{LCL}
Let $h: Y\rightarrow Z$ be a mapping between two pre-topological spaces $(Y, \tau)$ and $(Z, \upsilon)$. The mapping $h$ is {\it pre-continuous} from $Y$ to $Z$ if $h^{-1}(W)\in \tau$ for each $W\in\upsilon$.
\end{definition}

\begin{definition}\cite{LCL}
Let $h: Y\rightarrow Z$ be a bijection between two pre-topological spaces $Y$ and $Z$. If $h$ and $h^{-1}: Z\rightarrow Y$ are all pre-continuous, then we say that $h$ is a {\it pre-homeomorphic mapping}. We also say that $Y$ and $Z$ are {\it pre-homeomorphic}.
\end{definition}

\begin{definition}\cite{LCL}
A pre-topological space $(Z, \tau)$ is called a {\it $T_{0}$-space} if for any $y, z\in Z$ with $y\neq z$ there exists $W\in\tau$ such that $W\cap \{y, z\}$ is exact one-point set.
\end{definition}

\begin{definition}\cite{LCL}
A pre-topological space $(Z, \tau)$ is called a {\it $T_{1}$-space} if for any $y, z\in Z$ with $y\neq z$ there are $V, W\in\tau$ so that $V\cap\{y, z\}=\{y\}$ and $W\cap\{y, z\}=\{z\}$.
\end{definition}

\begin{definition}\cite{LCL}
A pre-topological space $(Z, \tau)$ is called a {\it $T_{2}$-space}, or a {\it Hausdorff space}, if for any $y, z\in Z$ with $y\neq z$ there are $V, W\in\tau$ so that $y\in V$, $z\in W$ and $V\cap W=\emptyset$.
\end{definition}

\begin{definition}\cite{LCL}
Let $Z$ be a $T_{1}$ pre-topological space. We say that $Z$ is a {\it $T_{3}$ pre-topological space}, or a {\it regular space}, if for every $z\in Z$ and every closed set $A$ of $Z$ with $z\not\in A$ there are open subsets $V$ and $O$ such that $V\cap O=\emptyset$, $z\in V$ and $A\subseteq O$.
\end{definition}

\begin{definition}\cite{LCL}
Let $Z$ be a $T_{1}$ pre-topological space. Then $Z$ is  a {\it $T_{3\frac{1}{2}}$ pre-topological space}, or a {\it completely regular pre-topological space}, or a {\it Tychonoff pre-topological space}, provide for each $z\in Z$ and each closed subset $C\subseteq Z$ with $z\not\in C$ there exists a pre-continuous mapping $r: Z\rightarrow I$ so that $r(z)=0$ and $r(x)=1$ for each $x\in C$.
\end{definition}

\begin{definition}\cite{LFC1}
Let $\mu$ be a family of non-empty subsets of $X\times X$ such that the following conditions are satisfied

\smallskip
(U1) for any $U\in\mu$, $\triangle\subseteq U$;

\smallskip
(U2) if $U\in\mu$, then $U^{-1}\in\mu$;

\smallskip
(U3) if $U\in\mu$, then there exist $V, W\in\mu$ such that $V\circ W\subseteq U$;

\smallskip
(U4) if $U\in\mu$ and $U\subset V\subseteq X\times X$, then $V\in\mu$;

\smallskip
(U5) if $\bigcap\mu=\bigtriangleup$.

\smallskip
The family $\mu$ is a {\it pre-uniform structure} of $X$ if $\mu$ satisfies (U1)-(U5), the pair $(X, \mu)$ is a {\it pre-uniform space} and the members of $\mu$ are called {\it entourage}.
\end{definition}

\begin{definition}\cite{LFC1}
Let $f$ be a mapping from a pre-uniform space $(X, \mu)$ to a pre-uniform space $(Y, \nu)$. We say that $f$ is {\it pre-uniformly continuous} if for each $F\in\nu$ there exists $M\in\mu$ such that $\phi(M)\subseteq F$, where $\phi: X\times X\rightarrow Y\times Y$ defined by $\phi(x, z)=(f(x), f(y))$ for each $(x, z)\in X\times X$.
\end{definition}

Let $G$ be a group, with the neutral element $e$, and let $N$ be a real-valued function on $G$. We say that $N$ is a {\it prenorm} on $G$ if the following conditions are satisfied for all $x, y\in G$:

\smallskip
(a) $N(e)=0$;

\smallskip
(b) $N(xy)\leq N(x)+N(y)$;

\smallskip
(c) $N(x^{-1})= N(x)$.

\smallskip
\section{Basic properties of pre-topological groups}
In this section, we give the concept of pre-topological groups and discuss some basic properties of pre-topological groups.

\begin{definition}
A {\it pre-topological group} $G$ is a group which is also a pre-topological space such that the multiplication mapping of $G\times G$ into $G$ sending $x\times y$ into
$x\cdot y$, and the inverse mapping of $H$ into $G$ sending $x$ into $x^{-1}$, are pre-continuous mappings.
\end{definition}

Now we give some examples of pre-topological group which are all not topological groups.

\begin{example}\label{d0}
(1) There exists a $T_{2}$ finite pre-topological group. Indeed, let $G=\{\overline{0}, \overline{1}, \overline{2}, \overline{3}, \overline{4}, \overline{5}\}$ be the set of surplus class with respect to module 6. We endow $G$ with a pre-topology as follows:

$$\tau=\{\emptyset, \{\overline{0}, \overline{3}\}, \{\overline{1}, \overline{4}\}, \{\overline{2}, \overline{5}\}, \{\overline{0}, \overline{2}, \overline{4}\},
\{\overline{1}, \overline{3}, \overline{5}\}, \{\overline{0}, \overline{1}, \overline{3}, \overline{4}\}, \{\overline{0}, \overline{2}, \overline{3}, \overline{5}\},
\{\overline{0}, \overline{2}, \overline{3}, \overline{4}\},$$$$\{\overline{0}, \overline{1}, \overline{3}, \overline{5}\}, \{\overline{1}, \overline{2}, \overline{4},
\overline{5}\}, \{\overline{0}, \overline{1}, \overline{2}, \overline{4}\}, \{\overline{1}, \overline{3}, \overline{4}, \overline{5}\}, \{\overline{0}, \overline{2},
\overline{4}, \overline{5}\}, \{\overline{1}, \overline{2}, \overline{3}, \overline{5}\}, \{\overline{0}, \overline{1}, \overline{2}, \overline{3}, \overline{4}\}, $$$$
\{\overline{0}, \overline{1}, \overline{3}, \overline{4}, \overline{5}\}, \{\overline{0}, \overline{2}, \overline{3}, \overline{4}, \overline{5}\}, \{\overline{0},
\overline{1}, \overline{2}, \overline{3}, \overline{5}\}, \{\overline{0}, \overline{1}, \overline{2}, \overline{4}, \overline{5}\}, \{\overline{1}, \overline{2},
\overline{3}, \overline{4}, \overline{5}\}, G\}.$$

It easily check that $(G, \tau)$ is a $T_{2}$ pre-topological group; however, it is not a topological group since $\{\overline{1}, \overline{4}\}\cap \{\overline{0},
\overline{2}, \overline{4}\}=\{4\}$ is not open in $G$.

(2) Let $G$ be the group $(\mathbb{R}, +)$ and endowed with a pre-topology which has a pre-base as follows:
$$\mathscr{B}=\{(-\infty, a): a\in\mathbb{R}\}\cup\{(b, +\infty): b\in\mathbb{R}\}.$$ Then $(\mathbb{R}, \tau)$ is a pre-topological group but not a topological group.

(3) Let $G$ be a group, and let $\tau_{1}$ and $\tau_{2}$ be two group topologies on $G$ such that $\tau_{1}\nsubseteq\tau_{2}$ and $\tau_{2}\nsubseteq\tau_{1}$. Let $\tau$
be the pre-topological group generated on $G$ by the family $\tau_{1}\cup\tau_{2}$. Then $(G, \tau)$ is a pre-topological group which is not a topological group.
\end{example}

In order to discuss the properties of pre-topological groups, we give some definitions.

\begin{definition}
Let $\varphi: G\rightarrow H$ be a mapping, where $G$ and $H$ are pre-topological groups. We say $\varphi$ is a {\it morphism} if $\varphi$ is pre-continuous and group
homomorphism. Further, we say that $\varphi$ is a {\it pre-topological isomorphism} if $\varphi$ is pre-homeomorphism and group homomorphism.
\end{definition}

\begin{definition}
A pre-topological space $X$ is said to be {\it pre-homogeneous} if for any $x, y\in X$ there exists a pre-homeomorphism $\varphi: X\rightarrow X$ such that $\varphi(x)=y$.
\end{definition}

The following two propositions and corollary are obvious.

\begin{proposition}
Let $G$ be a pre-topological group and $g\in G$. Then the left (right) translation mapping $L_{g} (R_{g}): G\rightarrow G$, defined by $L_{g}(x)=gx (R_{g}(x)=xg)$, is a
pre-homeomorphism.
\end{proposition}

\begin{corollary}
Each pre-topological group is a pre-homogeneous space.
\end{corollary}

Therefore, we have the following proposition.

\begin{proposition}\label{t1}
Let $G$ be a pre-topological group. If $\mathscr{B}$ is a local pre-base at $e$, then for each $g\in G$ the local pre-base at $g$ is equal to $\mathscr{B}_{g}=\{gU:
U\in\mathscr{B}\}$ or $\mathscr{B}_{g}=\{Ug: U\in\mathscr{B}\}$.
\end{proposition}

The following two theorems give the properties of an open base at the neutral element $e$ of $G$, which play important roles in the study of pre-topological groups.

\begin{theorem}\label{t0}
Let $G$ be a pre-topological group and $\mathscr{B}_{e}$ be a pre-base at the neutral element $e$ of $G$. Then the following statements hold:

\smallskip
(1) For each $U\in \mathscr{B}_{e}$, there exist $V, W\in\mathscr{B}_{e}$ such that $VW\subseteq U$.

\smallskip
(2) For each $U\in \mathscr{B}_{e}$, there exists $V\in\mathscr{B}_{e}$ such that $V^{-1}\subseteq U$.

\smallskip
(3) For each $U\in \mathscr{B}_{e}$ and any $g\in U$, there exists $V\in\mathscr{B}_{e}$ such that $Vg\subseteq U$.

\smallskip
(4) For each $U\in \mathscr{B}_{e}$ and any $g\in G$, there exists $V\in\mathscr{B}_{e}$ such that $gVg^{-1}\subseteq U$.
\end{theorem}

\begin{proof}
Assume $G$ is a pre-topological group. Properties (1) and (2) follow from the pre-continuity of the multiplication mappings $(x,y)\mapsto xy$ and $x\mapsto x^{-1}$ at the neutral element $e$, respectively. Property (3) follows from the pre-continuity of right translation mapping $R_{g}: x\mapsto xg$ in $G$. Property (4) follows from the  pre-continuity of right translation mapping $R_{g^{-1}}: x\mapsto xg^{-1}$ and left translation mapping $L_{g}: xg^{-1}\mapsto gxg^{-1}$ are pre-homeomorphism of $G$.
\end{proof}

\begin{theorem}\label{t2022}
Let $G$ be a group, and let $\mathscr{U}$ be a family of subsets of $G$ satisfying conditions (1)-(4) of Theorem~\ref{t0}. Then the family $\mathcal{B}_{\mathscr{U}}=\{Ua:
a\in G, U\in\mathscr{U}\}$ is a pre-base for a pre-topology $\tau$ on $G$ such that $(G, \tau)$ is a pre-topological group.
\end{theorem}

\begin{proof}
Let $\mathscr{U}$ be a family of subsets of $G$ satisfying conditions (1)-(4) of Theorem~\ref{t0}, and let
$$\tau=\{W\subset G: \mbox{for each}\ x\in W\ \mbox{there exists}\ U\in\mathscr{U}\ \mbox{such that}\ Ux\subseteq W\}.$$In order to prove that $(G, \tau)$ is a pre-topological group, we divide the proof into the following five claims.

\smallskip
{\bf Claim 1:} $\tau$ is a pre-topology on $G$.

\smallskip
Clearly, we have $\emptyset$, $G\in \tau$. Take any non-empty subfamily $\gamma$ of $\tau$. Pick any $x\in \bigcup \gamma$; then there exists $W\in \gamma$ such that $x\in W$, hence we can find $U\in \mathscr{U}$ so that $Ux\subseteq W$, that is, $Ux\subseteq W\subseteq \bigcup \gamma$, then $\bigcup\gamma\in \tau$. Therefore, $\tau$ is a pre-topology on $G$.

\smallskip
{\bf Claim 2:} The family $\mathcal{B}_{\mathscr{U}}=\{Ua:a\in G, U\in\mathscr{U}\}$ is a pre-base for the pre-topology $\tau$ on $G$.

\smallskip
We first prove that $\mathcal{B}_{\mathscr{U}}$ is a subfamily of $\tau$. It suffices to verify that for any $a\in G$ and $U\in \mathscr{U}$ we have $Ua\in \tau$. Take any $y\in Ua$; then $ya^{-1}\in U$. From condition (3), there exists $V\in \mathscr{U}$ such that $Vya^{-1}\subseteq U$, that is, $Vy\subseteq Ua$. Hence, $Ua\in \tau$. Now we prove that $\mathcal{B}_{\mathscr{U}}$ is a pre-base for $\tau$. Indeed, take any $W \in \tau$ and $a\in W$. Since $W\in\tau$, there exists $U\in \mathscr{U}$ such that $Ua\subseteq W$. Therefore, $\mathcal{B}_{\mathscr{U}}$ is a pre-base for $\tau$.

\smallskip
{\bf Claim 3:} The multiplication of $G$ is pre-continuous with respect to the pre-topology $\tau$.

\smallskip
Take any $x,y\in G$, and let $O$ be any element of $\tau$ such that $ab\in O$. Since $O\in\tau$, there exists $W\in \mathscr{U}$ such that $Wab\subseteq O$. In order to prove the pre-continuity of multiplication, it suffices to find $U, V\in \mathscr{U}$ such that $UaVb\subseteq Wab$, which is equivalent to  $UaVa^{-1}\subseteq W$. By condition(1), there exist $U, U^{\prime}\in \mathscr{U}$ such that  $UU^{\prime}\subseteq W$; then from condition(4) there exists $V\in \mathscr{U}$ such that $aVa^{-1}\subseteq U^{\prime}$, which implies that $UaVa^{-1}\subseteq UU^{\prime}\subseteq W$, thus $UaVb\subseteq Wab\subseteq O$. Hence, the multiplication in $G$ is pre-continuous with respect to the pre-topology $\tau$.

\smallskip
{\bf Claim 4:} For any $b\in G$ and $V\in \mathscr{U}$, we have $bV\in \tau$.

\smallskip
Take any $y\in bV$, we have to find $U\in \mathscr{U}$ such that $Uy\subseteq bV$. Clearly, we have $b^{-1}y\in V$. By condition (3), there exists $W\in \mathscr{U}$ such that $Wb^{-1}y\subseteq V$. Then from condition(4) there is $U\in \mathscr{U}$ such that $b^{-1}Ub\subseteq W$. Hence $b^{-1}Ubb^{-1}y\subseteq Wb^{-1}y\subseteq V$, that is, $b^{-1}Uy\subseteq V$, thus $Uy\subseteq bV$. Therefore, $bV\in \tau$.

\smallskip
{\bf Claim 5:} The inverse mapping $In$ of $G$ onto $G$ given by $In(x)=x^{-1}$ is pre-continuous with respect to the pre-topology $\tau$.

\smallskip
By Claims 2 and~4, it suffices to prove that $U^{-1} \in \tau$ for any $U\in \mathscr{U}$. Take any $x\in U^{-1}$; then $x^{-1}\in U$. From condition (3), there exists $V\in \mathscr{U}$ such that $Vx^{-1}\subseteq U$. Then from condition (2), there exists $W\in \mathscr{U}$ such that $W^{-1}\subseteq V$, hence we have $W^{-1}x^{-1}\subseteq Vx^{-1}\subseteq U$, that is, $xW\subseteq U^{-1}$. By Claim 4 again, $xW$ is an open neighbourhood of $x$ in $\tau $. Hence, $U^{-1}\in \tau$.

Therefore, we prove that $(G, \tau)$ is a pre-topological group and the family $\mathcal{B}_{\mathscr{U}}=\{Ua: a\in U, U\in\mathscr{U}\}$ is a pre-base of the pre-topology $\tau$.
\end{proof}

\begin{definition}
Let $G$ be a pre-topological space and $U$ be an open neighborhood of a point $b$ of $G$. We say that $U$ is an {\it atom} at $b$ if $W=U$ for any open neighborhood $W$ of $b$ with $W\subseteq U$.
\end{definition}

For a finite pre-topological group, we have the following proposition.

\begin{proposition}\label{p0}
Let $G$ be a finite pre-topological group and $\mathscr{B}_{e}$ be a pre-base at the identity $e$ of $G$. Then the following statements hold:

\smallskip
(1) For each $U\in \mathscr{B}_{e}$, there exist $V\in\mathscr{B}_{e}$ such that $V^{2}\subseteq U$.

\smallskip
(2) If $U\in \mathscr{B}_{e}$ is an atom at $e$, then $U^{-1}$ is an atom at $e$ and $U^{n}=U$ for each $n\in\mathbb{N}$.
\end{proposition}

\begin{proof}
(1) For each $U\in \mathscr{B}_{e}$, it follows from the finiteness of $G$ that there exists $V\in\mathscr{B}_{e}$ such that $V$ is an atom at $e$ and $V\subseteq U$.
Since $G$ is a pre-topological group, there exist $V_1, V_2\in\mathscr{B}_{e}$ such that $V_1V_2\subseteq V\subseteq U$.
Then, $V_1= V_2=V$ since $V$ is an atom at $e$. Hence, $V^{2}\subseteq U$.

(2) Take any $U\in \mathscr{B}_{e}$ such that $U$ is an atom at $e$. It follows from (2) of Theorem~\ref{t0} that there exists $V\in\mathscr{B}_{e}$ such that $V^{-1}\subseteq U$.
 Since $V^{-1}\in \mathscr{B}_{e}$ and $U\in \mathscr{B}_{e}$ is an atom at $e$, we have $V^{-1}=U$, that is, $U^{-1}=V$. Moreover, it easily verify that $V$ is an atom at $e$. Finally, we check that $U^{n}=U$ for each $n\in\mathbb{N}$.
From (1), it easily see that $V^{2}\subseteq V\subseteq U$,
then we have $V=U$ and $U^{2}=U$. By induction, assume that $U^{k}=U$ for $2\leq k\leq n$. For $k=n+1$, we have $U^{n+1}=U^{n}U=UU=U$. Hence, $U^{n}=U$ for each $n\in\mathbb{N}$.
\end{proof}

By Theorem~\ref{t0} and Proposition~\ref{p0}, it is natural to consider the following classes of pre-topological groups.

\begin{definition}
Let $G$ be a pre-topological and $\mathscr{B}_{e}$ be a pre-base at the identity $e$. Then

\smallskip
$\bullet$ we say that $G$ is {\it a symmetrically pre-topological group} if $\mathscr{B}_{e}$ is symmetric;

\smallskip
$\bullet$ we say that $G$ is {\it a strongly pre-topological group} if $\mathscr{B}_{e}$ satisfies (1) in Proposition~\ref{p0};

\smallskip
$\bullet$ we say that $G$ is {\it an almost topological group} if $G$ is a symmetrically and strongly pre-topological group.
\end{definition}

Clearly, each topological group is an almost topological group, each almost topological group is both a strongly pre-topological group and a symmetrically pre-topological group, and each strongly pre-topological group or symmetrically pre-topological group is a pre-topological group. The following example shows the inverse relations among them.

\begin{example}\label{eee}
(1) There exists a pre-topological group which is not a strongly pre-topological group.

\smallskip
(2) There exists a strongly pre-topological group which is not a symmetrically pre-topological group.

\smallskip
(3) There exists an almost topological group which is not a topological group.
\end{example}

\begin{proof}
(1) Let $G$ be the group $(\mathbb{R}, +)$ and endowed with a pre-topology which has a pre-base as follows:
$$\mathscr{B}=\{(-\infty, a): a\in\mathbb{Z}\}\cup\{(b, +\infty): b\in\mathbb{Z}\}.$$ Then $(\mathbb{R}, \tau)$ is a pre-topological group; however, it is easily checked that it not a strongly pre-topological group.

(2) Let $\mathbb{R}$ be the real number with usual addition `+'. Let $$\mathscr{U}=\{[a, b): a, b\in\mathbb{R}, b>a\}\cup\{(b, a]:
a, b\in\mathbb{R}, a>b\}.$$ Then $\mathscr{U}$ satisfies (1)-(4) of Theorem~\ref{t0} and (1) of Proposition~\ref{p0}. Hence the pre-topology $\tau$ generated by the family $\mathscr{U}$ is a strongly pre-topological
groups on $\mathbb{R}$. However, $G$ is not
a symmetrically pre-topological group.

(3) The pre-topological group in (1) of Example~\ref{d0} is just our requirement.
\end{proof}

However, the following question is still unknown for us.

\begin{question}
 Is each symmetrically pre-topological group a strongly pre-topological group?
\end{question}

The following proposition shows that the properties of an almost topological group is very similar to a topological group.

\begin{proposition}
Let $G$ be a pre-topological group. Then $G$ is an almost topological group if and only if for each open neighborhood $U$ of $e$ there exists an open neighborhood $V$ of $e$ such that $VV^{-1}\subseteq U$.
\end{proposition}

\begin{proof}
The necessity is obvious. In order to prove the sufficiency, it suffices to shows that $G$ has a symmetric pre-base. Take any $U$ in $\mathscr{B}_{e}$. Then we can find $V\in \mathscr{B}_{e}$ such that $VV^{-1}\subseteq U$, hence there exists $W\in \mathscr{B}_{e}$ such that $WW^{-1}\subseteq V$. Put $O=WW^{-1}$. Then $O$ is an open neighbour of $e$ and $O=O^{-1}$. Moreover, $O^{2}=O^{-1}O=(W^{-1}W)^{-1}(W^{-1}W)\subseteq VV^{-1}\subseteq U$. Hence, $G$ has a symmetric pre-base.
\end{proof}

The following theorem shows that the separation of $T_{0}$ is equivalent to the regularity in a pre-topological group.

\begin{theorem}\label{t2022111}
Each $T_{0}$ pre-topological group is regular.
\end{theorem}

\begin{proof}
Let $G$ be a $T_{0}$ pre-topological group. We first prove that $G$ is $T_{1}$. Indeed, take two distinct points $x$ and $y$. Since $G$ is $T_{0}$, without loss of
generality, we may assume that there exists an open neighborhood $U$ of $e$ such that $y\not\in xU$, hence $x\not\in yU^{-1}$. From (2) of Theorem~\ref{t0}, there exists an
open neighborhood $V$ of $e$ such that $V\subseteq U^{-1}$, then $x\not\in yV$. Therefore, $G$ is $T_{1}$. Next we prove that it is regular.

From the pre-homogeneous, it suffices to prove that for each open neighborhood $U$ of $e$ there exists an open neighborhood $V$ of $e$ such that $\overline{V}\subseteq U$.
From (1) of Theorem~\ref{t0}, there exist $V, W\in\mathscr{B}_{e}$ such that $VW\subseteq U$. We claim that $\overline{V}\subseteq U$. Indeed, take any $z\in \overline{V}$.
Then $zW^{-1}\cap V\neq\emptyset$, which implies that $z\in VW\subset U.$ Therefore, $\overline{V}\subseteq U$. Thus $G$ is regular.
\end{proof}

In the theory of topological groups, it is well known that any discrete subgroup in a topological group is closed. However, the situation is different in the theory of pre-topological groups. Indeed, the subgroup $H=\{\overline{0}, \overline{1}, \overline{5}\}$ is a discrete subgroup of the pre-topological group $G$ of (1) in Example~\ref{d0}. However, $H$ is not closed since $G=\overline{H}\neq H$. Further, $G$ is also an almost topological group. It is natural to pose the following question.

\begin{question}\label{q1}
Let $(G, \tau)$ be a pre-topological group. If $H$ is a discrete subgroup $G$, under what condition is $H$ closed in $G$?
\end{question}

In order to give some partial answers to Question~\ref{q1}, we introduce the following concepts.

\begin{definition}
Let $(G, \tau)$ be a pre-topological group. Then we say that $(G, \tau^{\star})$ is a {\it co-reflexion group topology} of $\tau$ if $\tau^{\star}$ is the coarsest topology on $G$ such that $\tau\subset\tau^{\star}$ and $(G, \tau^{\star})$ is a topological group; we say that $(G, \tau_{\star})$ is a {\it reflexion group topology} of $\tau$ if $\tau_{\star}$ is the strongest topology on $G$ such that $\tau_{\star}\subset\tau$ and $(G, \tau_{\star})$ is a topological group.
\end{definition}

Note that the reflexion group topologies of a pre-topological group are not unique. Moreover, we have the following proposition.

\begin{proposition}\label{p20221}
Let $(G, \tau)$ be a pre-topological group. Then the family $$\mathscr{B}=\{\bigcap\mathscr{F}: \mathscr{F}\subseteq\tau_{e}, |\mathscr{F}|<\omega\}$$ is an open neighborhood base at $e$ for the co-reflexion group topology $(G, \tau^{\star})$.
\end{proposition}

\begin{proposition}
Let $G$ be a pre-topological group. If $H$ is a discrete subgroup in $(G, \tau_{\star})$, then $H$ is closed in $G$.
\end{proposition}

\begin{proof}
Since $(G, \tau_{\star})$ is a topological group and $H$ is discrete in $(G, \tau_{\star})$, it follows from \cite[Corollary 1.4.18]{AT} that $H$ is closed in $(G, \tau_{\star})$, hence $H$ is closed in $(G, \tau)$ since $\tau_{\star}\subseteq\tau$.
\end{proof}

The following proposition is obvious.

\begin{proposition}
Let $G$ be a pre-topological group. If $H$ is a discrete subgroup in $(G, \tau)$, then $H$ is closed in $(G, \tau^{\star})$.
\end{proposition}

Assume that $U$ is a neighborhood of the neutral element of a pre-topological group $G$. A subset $B$ of $G$ is called {\it $U$-disjoint} if $b\not\in aU$, for any distinct $a, b\in B$.

\begin{lemma}\label{r1}
Let $G$ be an almost topological group, and let $U$ and $V$ be two open neighborhoods of the neutral element in $G$ such that $V^{4}\subset U$ and $V^{-1}=V$. If a subset $B$ of $G$ is $U$-disjoint, then the family of open sets $\{aV: a\in B\}$ is discrete in $G$.
\end{lemma}

\begin{proof}
It suffices to prove that, for every $x\in G$, the open neighbourhood $xV$ of $x$ intersects at most one element of the family $\{aV: a\in B\}$. Suppose not, then there exist $x\in G$ and distinct elements $a$, $b\in B$ such that $xV\cap aV \neq \varnothing$ and $xV\cap bV \neq \varnothing$. Then $x^{-1}a\in V^{2}$ and $b^{-1}x\in V^{2}$, hence $b^{-1}a =(b^{-1}x)(x^{-1}a)\in V^{4} \subseteq U$, thus $a\in bU$, which is a contradiction.
\end{proof}

Let $(G, \tau)$ be a pre-topological space. If each locally finite family of open subsets is finite, then we say that $G$ is {\it feebly compact}.

\begin{theorem}\label{s2}
Each discrete subgroup $H$ of a feebly compact almost topological group $G$ is finite.
\end{theorem}

\begin{proof}
Since $K$ is discrete in $G$, there exists an open neighbourhood $U$ of the neutral element in $G$ such that $U\cap H= \{e\}$. Then there is an open neighbourhood $V$ of $e$ such that $V^{4} \subseteq U$ because $G$ is an almost topological group. We claim that $H$ is $U$-disjoint. Indeed, take any distinct elements $h_{1}, h_{2}\in H$; then $h_{1}^{-1}h_{2}\notin \{e\}=U\cap H$. Since $h_{1}^{-1}h_{2}\in H$, it follows that $h_{1}^{-1}h_{2}\notin U$, that is, $h_{2}\notin h_{1}U$. Therefore, $H$ is $U$-disjoint. By Lemma~\ref{r1}, the family $\eta =\{hV: h\in H\}$ is discrete in $G$. Since $G$ is feebly compact, it follows that $H$ is finite.
\end{proof}

\begin{corollary}
Each infinite feebly compact almost topological group $G$ contains a non-closed countable subset.
\end{corollary}

\begin{proof}
Take any infinite countable subset $A$ of $G$, and let $H$ be the subgroup of $G$ algebraically generated by $A$. Then $H$ is countable and infinite. Therefore, by Theorem~\ref{s2}, $H$ cannot be discrete. Therefore, the subset $B=H\setminus \{e\}$ of $H$ is not closed in $H$, thus $B$ not closed in $G$.
\end{proof}

Finally, we discuss the connectedness \cite{LCL} of pre-topological groups. Let $G$ be pre-topological group with the neutral element $e$. The {\it connected component} of $G$ is the union of all connected subsets of $G$ containing $e$. It follows from \cite[Theorem 24]{LCL} that the connected component of $G$ can be described as the biggest connected subspace of $G$ containing $e$.

\begin{proposition}
The connected component $H$ of any pre-topological group $G$ is a closed invariant subgroup of $G$.
\end{proposition}

\begin{proof}
Since pre-continuity preserves the connectedness, it follows that $H^{2}$ and $H^{-1}$ are connected, hence $H$ is a group. From \cite[Theorem 25]{LCL}, $H$ is also closed. Now we only prove that $H$ is invariant. Since the left and right translations are pre-homeomorphism of $G$ onto itself, it follows that $aH$ and $Ha$ are connected for any $a\in G$, hence $aHa^{-1}$ and $a^{-1}Ha$ are also connected. For each $a\in G$, since $e\in aHa^{-1}\cap a^{-1}Ha$ and $H$ is the biggest connected subset of $G$ containing $e$, it follows that $aHa^{-1}\subseteq H$ and $a^{-1}Ha\subseteq H$, which implies that $H\subseteq aHa^{-1}$, hence $H=aHa^{-1}$.
\end{proof}

\begin{proposition}\label{z1}
Let $U$ be an arbitrary open neighborhood of the neutral $e$ of a connected pre-topological group $G$. Then $G=\bigcup_{n=1}^{\infty}(U\cup U^{-1})^{n}$.
\end{proposition}

\begin{proof}
Let $U$ be an arbitrary open neighborhood of the neutral $e$. Then $H=\bigcup_{n=1}^{\infty}(U\cup U^{-1})^{n}$ is an open subgroup of $G$. According to Proposition~\ref{s1} below, $H$ is closed in $G$. Since $G$ is connected, we have $H=G$.
\end{proof}

\begin{theorem}\label{ttttt3}
Let $K$ be a discrete invariant subgroup of a connected pre-topological group $G$. If, for any $x\in G$ and open set $U$ with $x\in U$, there exists an open neighborhood $V$ of $e$ such that $VxV\subset U$, then $K$ is contained in the center of the pre-topological group $G$.
\end{theorem}

\begin{proof}
If $K={e}$, then it is obvious. Therefore, suppose that the subgroup $K$ is not trivial. Take any $x\in K\setminus\{e\}$. Since the group $K$ is discrete, there is an open neighbourhood $U$ of $x$ such that $U\cap K=\{x\}$. Then it follows from our assumption that there exists an open neighborhood $V$ of $e$ such that $VxV\subseteq U$. For any $y\in V$, since $K$ is an invariant subgroup of $G$, we have that $yxy^{-1}\in K$; it is obvious that $yxy^{-1}\in VxV\subseteq U$. Thus, $yxy^{-1}\in U\cap K=\{x\}$, that is, $yxy^{-1}=x$. Thus $yx=xy$ and $y^{-1}x=xy^{-1}$ for each $y\in V$. This shows that $x$ commutes with every element of $V\cup V^{-1}$. Next we prove that $x$ commutes with every element of $G$.

Since $G$ is connected, it follows from Proposition~\ref{z1} that $G=\bigcup_{n\in\mathbb{N}}(V\cup V^{-1})^{n}$. Then, for every $g\in G$, $g$ can be written in the form $g=y_{1}y_{2}\cdots y_{n}$, where $y_{1}y_{2}\cdots y_{n}\in V\cup V^{-1}$ and $n\in \mathbb{N}$. Since $x$ commutes with every element of $V\cup V^{-1}$, we conclude that $$gx=y_{1}y_{2}\cdots y_{n}x=y_{1}y_{2}\cdots y_{n-1}xy_{n-1}=\cdots=xy_{1}y_{2}\cdots y_{n}=xg.$$ Since $x$ is an arbitrary element of $K$, we conclude that the center of $G$ contains $K$.
\end{proof}

From Theorem~\ref{ttttt3}, we have the following question.

\begin{question}
Let $G$ be a strongly pre-topological group (almost topological group). For any $x\in G$ and open set $U$ with $x\in U$, does there exist an open neighborhood $V$ of $e$ such that $VxV\subset U$?
\end{question}

\section{Generalizations of pre-topological groups}

In this section, we give some generalizations of pre-topological group, such as, semi-pre-topological group, quasi-pre-topological group and para-pre-topological group. Moreover, we study some basic properties of them.

\begin{definition}
Let $\tau$ be a pre-topology on a group $G$. We say that $(G, \tau)$ is a

\smallskip
$\bullet$ right (left) pre-topological group if for each $a\in G$, the right (left) action $R_a$ $(L_a)$ of $a$ on $G$ is a pre-continuous mapping of the space $G$ to itself.

\smallskip
$\bullet$ semi-pre-topological group if $(G, \tau)$ are both right pre-topological group and left pre-topological group.

\smallskip
$\bullet$ quasi-pre-topological group if $(G, \tau)$ is a semi-pre-topological group such that the inverse mapping $In:G\rightarrow G$ is pre-continuous.

\smallskip
$\bullet$ para-pre-topological group if the multiplication mapping $G \times G\rightarrow G $ is pre-continuous, where $G \times G$ is given the product pre-topology.
\end{definition}

\begin{definition}
Let $G$ be a semigroup, and let $(G, \tau)$ be a pre-topological space. We say that $(G, \tau)$ is a

\smallskip
$\bullet$ right (left) pre-topological semigroup if for each $a\in G$, the right (left) action $R_a$ $(L_a)$ of $a$ on $G$ is a pre-continuous mapping of the space $G$ onto itself.

\smallskip
$\bullet$ semi-pre-topological semigroup if $(G, \tau)$ both right pre-topological semigroup and left pre-topological semigroup.

\smallskip
$\bullet$ pre-topological semigroup if the multiplication mapping $G \times G\rightarrow G$  is pre-continuous, where $G \times G$ is endowed with the product pre-topology.
\end{definition}

The following examples shows that the class of pre-topological groups is strictly contained in the class of quasi-pre-topological groups and the class of para-pre-topological groups respectively.

\begin{example}\label{e0}
There exists a finite quasi-pre-topological group $G$ such that $G$ is not a pre-topological group.
\end{example}

\begin{proof}
Indeed, let $G=\{\overline{0}, \overline{1}, \overline{2}, \overline{3}\}$ be the set of surplus class with respect to module 4. We endow $G$ with a pre-topology as follows:
$$\tau=\{\emptyset, \{\overline{0}, \overline{1}, \overline{3}\}, \{\overline{0}, \overline{1}, \overline{2}\}, \{\overline{1}, \overline{2}, \overline{3}\}, \{\overline{0},
\overline{2}, \overline{3}\}, G\}.$$It easily check that $(G, \tau)$ is a quasi-pre-topological group. However, $(G, \tau)$ is not a pre-topological group since the
multiplication of $(G, \tau)$ is not jointly pre-continuous (Indeed, for any open neighborhoods $U, V$ of $e$ we have $\overline{0}\in UV$, hence $UV\nsubseteq\{\overline{0},
\overline{1}, \overline{3}\}$).
\end{proof}

\begin{example}
There exists a para-pre-topological group $G$ such that $G$ is not a pre-topological group and para-topological group.
\end{example}

\begin{proof}
Let $\mathbb{R}$ be the real number with usual addition `+'. Let $\tau_{1}$ be the topology generated by the family $\{[x, x+y): x\in\mathbb{R}, y\in\mathbb{R}^{+}\}$ (that
is, Sorgenfrey line), and let $\tau_{2}$ be the topology generated by the family $\{\{x\}\cup \{m+x: m\geq n\}: n\in\mathbb{N}\}$. It easily check that both $(G, \tau_{1})$
and $(G, \tau_{2})$ are para-topological groups on $\mathbb{R}$. Clearly, $\tau_{1}\nsubseteq\tau_{2}$ and $\tau_{2}\nsubseteq\tau_{1}$. Let $(G, \tau)$ be the pre-topology
which is generated by the family $\tau_{1}\cup\tau_{2}$ as follows
$$\tau=\left\{U\subset \mathbb{R}:\ \mbox{there exist}\ \mathcal{U}_{1}\subseteq\tau_{1}\ \mbox{and}\ \mathcal{U}_{2}\subseteq\tau_{2}\ \mbox{such that}\
U=(\bigcup\mathcal{U}_{1})\cup\bigcup(\mathcal{U}_{2})\right\}.$$ Since $\{0\}$ is not open in $(G, \tau)$, it follows that $(G, \tau)$ is not a para-topological group.
Moreover, for any open neighborhood $W$ of $e$ in $\tau$, we have $-W\nsubseteq [0, 1)$, hence $G$ is not a pre-topological group.
\end{proof}

Indeed, we have the following proposition which is a machine to generate a para-pre-topological group from a para-topological group.

\begin{proposition}
Let $(G, \tau)$ be a para-topological group which is not a topological group. Then the pre-topology $\sigma$ on $G$, which has a pre-base $\tau\cup\tau^{-1}$, is a para-pre-topological group and is not a para-topological group, where $\tau^{-1}=\{U^{-1}: U\in \tau\}$.
\end{proposition}

\begin{proof}
It easily check that $(G, \sigma)$ is a para-pre-topological group. However, since $(G, \tau)$ is not a topological group, there exists $U\in\tau_{e}$ such that $U\cap U^{-1}\not\in\tau\cup\tau^{-1}$, hence $(G, \sigma)$ is not a para-topological group.
\end{proof}

Hence it is natural to pose the following question.

\begin{question}\label{q0}
When is a para-pre-topological group (or quasi-pre-topological group) pre-topological group?
\end{question}

The next theorem gives a partial answer to Question~\ref{q0}. First, we introduce a concept.

\begin{definition}
Let $(X, \tau)$ be a pre-topological space. The space $X$ is {\it locally finite} \cite{KE} if each point of $X$ has an open neighborhood which is a finite set.
\end{definition}

Clearly, each finite pre-topological space is locally finite.

\begin{theorem}\label{r0}
If $(G, \tau)$ is a locally finite para-pre-topological group, then $G$ is an almost topological group.
\end{theorem}

\begin{proof}
Clearly, since $G$ is locally finite, each open set in $G$ must contain a minimally open set. By the proof of (1) in Proposition~\ref{p0}, it suffices to prove that $U=U^{-1}$ for any minimally open set $U$ at $e$. If $G$ is discrete, then it is obvious. Hence we assume that $G$ is not discrete.
Take any minimally open set $U$ at $e$, and pick any $x\in U\setminus\{e\}\neq\emptyset$. Clearly, $U$ is finite. Then it follows from the pre-continuous of the multiplication of $G$ that we have
$U^{2}=U$, hence $xU\subset U$. Enumerate $U$ as $\{e, x_{1}, x_{2},\ldots ,x_{n}\}$. Since $|xU|=|U|$ and $xU\subset U$, we conclude that $xU=U$, then there exists $i\leq n$
such that $xx_{i}=e$, which implies that $x^{-1}=x_{i}\in U$. Therefore, $U=U^{-1}$. Thus $G$ is an almost topological group.
\end{proof}

\begin{remark}
From Example~\ref{e0}, there exists a finite quasi-pre-topological group which is not a pre-topological group. Moreover, R. EIlis proved that each locally compact Hausdorff
semi-topological group is a topological group. It natural to pose the following question.
\end{remark}

A pre-topological space $G$ is said to be {\it compact} if each open cover of $G$ has a finite subcover; $G$ is said to be {\it locally compact} if for each $x\in G$ there exists a compact neighborhood $U$ of $x$.

\begin{question}
If $G$ is a (locally) compact Hausdorff para-pre-topological group, then is $G$ a pre-topological group?
\end{question}

\begin{proposition}\label{s0}
Let $(G, \tau)$ be a right pre-topological group and $g$ be any element of $G$. Then the following statements hold:

\smallskip
(1) The right translation $r_{g}$ of $G$ by $g$ is a pre-homeomorphism of the pre-topological space $G$ onto itself.

\smallskip
(2) For any pre-base $\mathcal{B}_{e}$ of $G$ at $e$, the family $\mathcal{B}_{g}=\{Ug: U\in\mathcal{B}_{e}\}$ is a pre-base of $G$ at $g$.
\end{proposition}

\begin{proof}
Clearly, (2) follows from (1). Hence we only need to prove (1). For any $g\in G$, $r_{g}$ and $r_{g^{-1}}$ are all pre-continuous bijection; moreover, since $r_{g}r_{g^{-1}}=e$, we have $r_{g^{-1}}={r_{g}}^{-1}$. Therefore, $r_{g}$ and $(r_{g})^{-1}$ are all pre-continuous bijection, hence $r_{g}$ is a pre-homeomorphism of the pre-topological space $G$ onto itself.

\end{proof}

\begin{corollary}
In each semi-pre-topological group $G$ all, right and left, translations are pre-homeomorphism.
\end{corollary}

\begin{corollary}
If a subgroup $H$ of a right (or left) pre-topological group $G$ contains a non-empty open subset of $G$, then $H$ is open in $G$.
\end{corollary}

\begin{proof}
Suppose $U$ is a non-empty open subset of $G$ such that $U\subseteq H$. By Proposition~\ref{s0}, the set $R_a(U)=Ua$ is also open in $G$ for each $a\in H$. Hence, the set $H=\bigcup_{a\in H}Ua$ is open in $G$.
\end{proof}

\begin{proposition}
If $f: G\rightarrow H$ is a homomorphism of left (right) pre-topological groups and $f$ is pre-continuous at the neutral element $e$ of $G$, then $f$ is pre-continuous.
\end{proposition}

\begin{proof}
Take any $x\in G$, and let $O$ be an open neighbourhood of $y=f(x)$ in $H$. Since the left translation $L_{y}$ is a pre-homeomorphism of $H$, there exists an open neighbourhood $V$ of the neutral element $e_{H}$ in $H$ such that $yV\subseteq O$. By the pre-continuity of $f$ at $e_{G}$, it follows that there exists an open neighbourhood $U$ of $e_{G}$ in $G$ such that $f(U)\subseteq V$. Since $L_{x}$ is a pre-homeomorphism of $G$ onto itself, the set $xU$ is an open neighbourhood of $x$ in $G$; therefore, we can conclude that $f(x)\in f(xU)=f(x)f(U)=yf(U)\subseteq yV\subseteq
O$. Hence, $f$ is pre-continuous.
\end{proof}

\begin{theorem}\label{s1}
Each open subgroup $H$ of a right (left) pre-topological group $G$ is closed in $G$.
\end{theorem}

\begin{proof}
The family $\gamma=\{Ha: a\in G\}$ is all the right cosets of $H$ in $G$. Since each right translation is pre-homeomorphism mapping, it follows that $Ha$ is open in $G$ for each $a\in G$. Then the family $\gamma$ is a disjoint open cover of $G$, this implies that $H=G\setminus(\bigcup_{a\in G\setminus\{e\}}Ha)$. Thus, $H$ is closed in $G$.
\end{proof}

\begin{corollary}
Each right (left) pre-topological group $G$ is a pre-homogeneous pre-topological space.
\end{corollary}

\begin{proof}
Take any elements $x$ and $y$ in $G$ and put $z=x^{-1}y$; then $R_{z}(x)=xz=xx^{-1}y=y$. Hence each right pre-topological group $G$ is a pre-homogeneous pre-topological space. A similar argument can be applied in the case of left pre-topological group.
\end{proof}

\begin{proposition}
If $G$ is a left (right) pre-topological group and $U$ is an open subset of $G$, then the set $AU$ (resp., $UA$) is open in $G$ for any subset $A$ of $G$.
\end{proposition}

\begin{proof}
By (1) of Proposition~\ref{s0}, every right translation of $G$ is pre-homeomorphism of the pre-topological space $G$ onto itself, then $UA=\bigcup_{a\in A}Ua=\bigcup_{a\in A}r_{a}(U)$. Hence $UA$ is open in $G$. By a similar method, we can prove that $AU$ is open in $G$ when $G$ is a left pre-topological group.
\end{proof}

\begin{corollary}
Let $G$ be a semi-pre-topological group and $U$ an open subset of $G$. Then the sets $UA$ and $UA$ are open for any subset $A$ of $G$.
\end{corollary}

Finally, we mainly discuss the closure of a subset of a left (right) pre-topological group $G$. The following proposition gives an intimate relationship between the sets $AU$ or $UA$ and the closure operation, where $U$ is open in $G$ and $A$ is a subset of $G$.

\begin{proposition}\label{u0}
Let $G$ be a left (right) pre-topological group with pre-continuous inverse. Then, for each subset $A$ of $G$ and each open neighborhood $U$ of $e$, we have
$\overline{A}\subset AU$ (resp., $\overline{A}\subset UA$) .
\end{proposition}

\begin{proof}
Take any subset $A$ of $G$ and any open neighborhood $U$ of $e$. Since the inverse is pre-continuous, there exists an open neighbourhood $V$ of $e$ such that $V^{-1}\subseteq U$. Now take any $x\in \overline{A}$; since $xV$ is an open neighbourhood of $x$, we have $xV\cap A\neq\emptyset$, then there are $a\in A$ and $b\in V$ such that $a=xb$. Then $x=ab^{-1}\in AV^{-1}\subseteq AU$; hence, $\overline{A}\subseteq AU$.
\end{proof}

\begin{proposition}
Let $G$ be a left (right) pre-topological group with pre-continuous inverse, and let $\mathscr{B}_{e}$ a pre-base of $G$ at the neutral element $e$. Then, for each subset $A$ of $G$, we have $\overline{A}=\bigcap\{AU: U\in\mathscr{B}_{e}\}$ (resp., $\overline{A}=\bigcap\{UA: U\in\mathscr{B}_{e}\}$ ) .
\end{proposition}

\begin{proof}
By Proposition~\ref{u0}, we have $\overline{A}\subseteq AU$ for each $U\in\mathscr{B}_{e}$, then $\overline{A}\subseteq \bigcap\{AU: U\in\mathscr{B}_{e}\}$. Now assume that $x\not\in\overline{A}$. We prove that there exists $U\in\mathscr{B}_{e}$ such that $x\not\in AU$. Since $x\not\in\overline{A}$, there exists an open neighbourhood $W$ of $e$ such that $(xW)\cap A= \emptyset$. From the pre-continuity of the inverse, there is an open neighbourhood $U$ of $e$ such that $U^{-1}\subseteq W$, then $(xU^{-1})\cap A =\emptyset$, that is $x\not\in AU$; hence, $x\not\in\bigcap\{AU: U\in\mathscr{B}_{e}\}$. Therefore, $\bigcap\{AU: U\in\mathscr{B}_{e}\}\subseteq \overline{A}$. Then $\overline{A}= \bigcap\{AU: U\in\mathscr{B}_{e}\}$.

Similarly, the equality $\overline{A}=\bigcap\{UA: U\in\mathscr{B}_{e}\}$ holds for the right pre-topological groups with pre-continuous inverse.
\end{proof}

\begin{proposition}\label{v0}
Let $G$ be a semi-pre-topological group such that for every open set $U$ with $e\in U$, there exists an open neighborhood $V$ of $e$ with $V^{-1}\subset U$. Then $G$ is a quasi-pre-topological group.
\end{proposition}

\begin{proof}
Let $W$ be an any open set. We have to show that $W^{-1}$ is also open in $G$. Take any $x\in W^{-1}$. Then $x^{-1}\in W$. Since $G$ is a left pre-topological group, there exists an open neighbourhood $U$ of $e$ such that $x^{-1}U\subseteq W$. By the assumption, there exists an open neighbourhood $V$ of $e$ such that $V^{-1}\subseteq U$. Then ${(Vx)}^{-1}=x^{-1}V^{-1}\subseteq x^{-1}U\subseteq W$. Hence, the inverse mapping on $G$ is pre-continuous, thus $G$ is a quasi-pre-topological group.
\end{proof}

\begin{proposition}
Let $G$ be a semi-pre-topological group such that for each closed subset $A$ of $G$  and each point $x\not\in A$, we have $x\not\in AU$ for some open neighborhood $U$ of $e$.
Then $G$ is a quasi-pre-topological group.
\end{proposition}

\begin{proof}
By proposition~\ref{v0}, it suffices to verify that, for each open set $U\in\mathscr{B}_{e}$, there exists an open neighbourhood $V$ of $e$ such that $V^{-1}\subseteq U$. Take any open neighbourhood $U$ of $e$, and put $A=G\backslash U$; then $e\not\in A$. By the assumption, there exists an open neighbourhood $V$ of $e$ such that $e\not\in AV$, then $A\cap V^{-1}=\emptyset$, that is, $V^{-1}\subseteq G\backslash A=U$. Hence, $G$ is a quasi-pre-topological group.
\end{proof}

\begin{proposition}
Each left pre-topological group with pre-continuous inverse is a right pre-topological group and hence, a quasi-pre-topological group.
\end{proposition}

\begin{proof}
We can get from $x$ to $xa$ in three steps: from $x$ to $x^{-1}$, then from $x^{-1}$ to $a^{-1}x^{-1}$, and finally from $a^{-1}x^{-1}$ to $xa$. This means that $R_{a}=In\circ L_{a^{-1}}\circ In$. Since all the mappings on the right side of the above equality are pre-continuous, the right translation is pre-continuous.
\end{proof}

\begin{proposition}\label{w0}
Let $G$ be a semi-pre-topological semigroup, and $H$ a subsemigroup of $G$. Then the closure $\overline{H}$ of $H$ is a subsemigroup of $G$.
\end{proposition}

\begin{proof}
First, we prove that $\overline{H}\cdot H\subseteq \overline{H}$. Take any $y\in \overline{H}$ and $x\in H$. Since right translation is  pre-continuous ,we have $R_{x}(\overline{H})\subseteq \overline{R_{x}(H)}$ by \cite[Theorem 7]{LCL}, that is, $\overline{H}x\subseteq \overline{Hx}$. It follows from $x\in H$ and $H$ is a subsemigroup of $G$ that $Hx\subseteq H$, then, $yx\in \overline{H}x \subseteq \overline{Hx} \subseteq \overline{H}$, hence, $\overline{H}\cdot H\subseteq \overline{H}$. And then we can prove that $\overline{H} \cdot \overline{H}\subseteq \overline{H}$. Since each left translation is pre-continuous, we also have $y\overline{H}\subseteq \overline{yH}$ by \cite[Theorem 7]{LCL}. From $yH\in \overline{H}\cdot H\subseteq \overline{H}$, it follows that $y\overline{H}\subseteq \overline{yH}\subseteq \overline{\overline{H}}=\overline{H}$ by \cite[Theorem 7]{LCL}. Now, take any $z\in \overline{H}$, we have $yz\in \overline{H}$. Thus, $\overline{H}$ is a subsemigroup of $G$.
\end{proof}

\begin{proposition}\label{y0}
Let $G$ be an abstract group with a pre-topology $\tau$ such that the inverse mapping is pre-continuous. Then, for any symmetric subset $A$ of $G$, the closure of $A$ in $G$ is also symmetric.
\end{proposition}

\begin{proof}
Since the inverse mapping $In$ is pre-continuous and the composition $In\circ In$ is the identity mapping of $G$ onto itself. Thus, $In$ is a pre-homeomorphism, this implies that $\overline{In(A)} =In(\overline A)$ by \cite[Theorem 7]{LCL}, that is, $\overline{A^{-1}}=(\overline{A})^{-1}$. Since $A^{-1}= A$, then $\overline{A^{-1}} = \overline{A} =(\overline{A})^{-1}$, hence, the closure of $A$ in $G$ is also symmetric.
\end{proof}

\begin{proposition}\label{p202266}
Let $G$ be a quasi-pre-topological group, and $H$ an algebraic subgroup of $G$. Then the closure of $H$ is also a subgroup of $G$.
\end{proposition}

\begin{proof}
Take any $x$, $y\in \overline{H}$. From proposition~\ref{w0}, we have $xy\in \overline{H}$. By proposition~\ref{y0}, we conclude that $\overline{H}^{-1}= \overline{H}$. Hence, the closure of $H$ is also a subgroup of $G$.
\end{proof}

\begin{corollary}
Let $G$ be a pre-topological group, and $H$ a subgroup of $G$. Then the closure of $H$ is also a subgroup of $G$.
\end{corollary}

\section{Quotients of pre-topological groups}

In this section, we mainly discuss the quotient spaces of pre-topological groups and give the three isomorphisms of quotient spaces. Indeed, the authors in \cite{HK2013} has discuss the quotients of generalized topological groups. However, their paper has some gaps in some results, hence we systematically discuss the quotient of pre-topological groups and some results given without proofs. First, we need the following result.

\begin{theorem}\label{a4}
Suppose that $G$ is a semi-pre-topological group with the neutral element $e$ and a pre-topology $\tau$, and $H$ is a closed subgroup of $G$. Denote by $G/H$ the set of all left cosets $aH$ of $H$ in $G$, and endow it with the quotient pre-topology with respect to the canonical mapping $\pi$: $G \rightarrow G/H$ defined by $\pi(a)= aH$, for each $a\in G$. Then the family $\{\pi(xU): U\in \tau, e\in U\}$ is a local pre-base of the space $G/H$ at the point $xH\in G/H$, the mapping $\pi$ is open, and $G/H$ is a pre-homogeneous $T_{1}$-space.
\end{theorem}

\begin{proof}
Clearly, we have $\pi(xU)=\pi(xUH)$ for any $x\in G$ and $U\in\tau$ with $e\in U$, and the set $xUH$ is the union of a family of left cosets $yH$, where each $y\in xU$. Hence, $\pi^{-1}\pi (xUH)= xUH$. Since the set $xUH$ is open in $G$ and the mapping $\pi$ is quotient, it follows that $\pi (xUH)$ is open in $G/H$. Therefore, the mapping $\pi$ is open.

Now we prove that the family $\{\pi(xU): U\in \tau, e\in U\}$ is a local pre-base of the space $G/H$ at the point $xH\in G/H$. Take any open neighbourhood $W$ of $xH$ in $G/H$ and put $O=\pi^{-1}(W)$; obviously, we have $x\in O$. Since $\pi $ is pre-continuous, $O$ is open in $G$. Therefore, there exists an open neighbourhood $U$ of $e$ in $G$ such that $xU\subseteq O$, then $\pi(xU)\subseteq W$ and $\pi^{-1}\pi (xU)\subseteq O$. Since $xUH=\pi^{-1}\pi (xU)$, it follows that $\pi (xUH)=\pi (xU)\subseteq W$.

Finally, we prove $G/H$ is pre-homogeneous. Let $a$ be an arbitrary elements of $G$; we define a mapping $h_{a}$ of $G/H$ to itself by the rule $h_{a}(xH)=axH$. Clearly, the mapping is well defined. Since $G$ is a group, it easily check that $h_{a}$ is a bijection of $G/H$ onto $G/H$. Next we prove $h_{a}$ is pre-homeomorphism.

Take any $xH\in G$ and any open neighbourhood $U$ of $e$; then $\pi(xUH)$ is a neighbourhood of $xH$ in $G/H$. Similarly, the set $\pi(axUH)$ is a neighbourhood of $axH$ in $G/H$. Since $h_{a}(\pi(xUH))=\pi (axUH)$ and $h_{a}^{-1}(\pi(axUH))=\pi (xUH)$, then the mapping $h_{a}$ is open and pre-continuous. Thus, $h_{a}$ is pre-homeomorphism. For any $xH$, $yH$ in $G/H$, take $a=yx^{-1}$; then $h_{a}(xH)=yx^{-1}xH=yH$. Therefore, the quotient space $G/H$ is pre-homogeneous.

For any $xH$ in $G/H$, we have $\pi^{-1}(xH)=xH$. Since all left cosets are closed in $G$ and the mapping $\pi$ is quotient, it follows that $xH$ is closed in $G/H$. Hence, $G/H$ is $T_{1}$-space.
\end{proof}

\begin{proposition}
Suppose that $G$ be a pre-topological group, and $H$ is a closed subgroup of $G$, $\pi$ is the natural quotient mapping of $G$ onto the left quotient space $G/H$, $a\in G$, $L_a$ is the left translation of $G$ by $a$ (that is, $L_a(x)=ax$, for each $x\in G$), and $h_a$ is the left translation of $G/H$ by $a$ (that is, $h_a(xH)=axH$, for each $xH\in G/H$). Then $L_a$ and $h_a$ are homeomorphisms of $G$ and $G/H$, respectively, and $\pi\circ L_a=h_a\circ \pi$.
\end{proposition}

If $G$ is a left pre-topological group and $H$ is a closed invariant subgroup of $G$, then each left coset of $H$ in $G$ is also a right coset of $H$ in $G$, and a natural multiplication of cosets in $G/H$ is defined by the rule $xHyH=xyH$, for all $x,y\in G$. This operation turn $G/H$ into a group.

\begin{theorem}
Suppose that $G$ is a semi-pre-topological group with the neutral element $e$ and $H$ is a closed subgroup of $G$. Then $G/H$ with the quotient pre-topology and multiplication is a semi-pre-topological group, and the canonical mapping $\pi: G \rightarrow G/H$ is an open pre-continuous homomorphism. If $G$ is a pre-topological group and $H$ is invariant, then $G/H$ is a pre-topological group.
\end{theorem}

\begin{proof}
By theorem ~\ref{a4}, we have that the canonical mapping $\pi$ is open and pre-continuous. The mapping $\pi$ is also a homomorphism, since $\pi(ab)=abH=aHbH=\pi(a)\pi(b)$, for any $a, b\in G$.
\end{proof}

\begin{theorem}
Let $G$ be a pre-topological group and $H$ is a closed subgroup of $G$. Then $G/H$ with the quotient pre-topology is discrete if and only if $H$ is open in $G$.
\end{theorem}

\begin{proof}
Suppose that $H$ is open in $G$, then {$aH$} is also open in $G$. Since $\pi$ is a quotient mapping and $\pi(aH)= aH$, it follows that $\{aH\}$ is open in $G/H$, so $G/H$ is discrete.
Conversely, suppose that $G/H$ is discrete, then for every $x\in G$, $\pi(x)= xH\in G/H$, hence $\{xH\}$ is open in $G/H$. Since $\pi$ is a quotient mapping and $\pi^{-1}(xH)= xH\subseteq G$, we conclude that $xH$ is open in $G$, hence $H$ is also open in $G$.
\end{proof}

\begin{lemma}\label{e4}
Suppose that $G$ is a pre-topological group, $H$ is closed subgroup of $G$ and $\pi$ is the natural quotient mapping of $G$ onto the quotient space $G/H$. If $U$, $V$ and $W$ are open neighbourhoods of the neutral element $e$ in $G$ such that $WV\subseteq U$, then $\overline{\pi (V)}\subseteq \pi(U)$.
\end{lemma}

\begin{proof}
Take any $x$ in $G$ such that $\pi(x)\in \overline{\pi(V)}$, we only have to prove $\pi(x) \in \pi(U)$. Since $W^{-1}x$ is an open neighbourhood of $x$ and the mapping $\pi$ is open, we conclude that $\pi(W^{-1}x)$ is an open neighbourhood of $\pi(x)$. Hence, $\pi(W^{-1}x) \cap \pi(V)\neq \varnothing$, then there exists $a\in W^{-1}$ and $b\in V$ such that $\pi(ax)= \pi(b)$, which implies that $ax=bh$, for some $h\in H$. Therefore, $x=(a^{-1}b)h \in (WV)H\subseteq UH$, that is, $\pi(x) \in \pi(UH)=\pi(U)$.
\end{proof}

\begin{theorem}
For any pre-topological group $G$ and any closed subgroup $H$ of $G$, the quotient space $G/H$ is regular.
\end{theorem}

\begin{proof}
Let $\pi$ be the natural quotient mapping of $G$ onto the quotient space $G/H$ and $W$ be an arbitrary open neighbourhood of $\pi(e)$ in $G/H$, where $e$ is the neutral element of $G$. Since $\pi$ is pre-continuous, there exists an open neighbourhood $U$ of $e$ in $G$ such that $\pi(U)\subseteq W$. Let $V$ and $W$  are open neighbourhoods of $e$ such that $WV\subseteq U$. From Lemma~\ref{e4}, it follows that $\overline{\pi(V)}\subseteq \pi(U)\subseteq W$. Hence $G/H$ is regular at the point $\pi(e)$ by Theorem~\ref{a4}.
\end{proof}

\begin{proposition}\label{g4}
Suppose that $G$, $H$ and $K$ are abstract groups, and suppose that $\varphi: G\rightarrow H$ and $\psi: G\rightarrow K$ are homomorphisms such that $\psi(G)= K$ and $ker\psi\subseteq ker\varphi$. Then there exists a homomorphism $f: K\rightarrow H$ such that $\varphi =f\circ \psi$. In addition, if $G$, $H$, $K$ are pre-topological groups, then $\varphi$ and $\psi$ are pre-continuous, and if for each neighbourhood $U$ of the neutral element $e_{H}$ in $H$ there exists a neighbourhood $V$ of the identity $e_{K}$ in $K$ such that $\psi^{-1}(V)\subseteq \varphi^{-1}(U)$, then $f$ is pre-continuous.
\end{proposition}

\begin{proof}
The algebraic part of the proposition is well known. It only to verify that $f$ is pre-continuous. Let $U$ be a neighbourhood of $e_H$ in $H$. By our assumption there exists a neighbourhood $V$ of the neutral element $e_K$ in $K$ such that $W=\psi^{-1}(V)\subseteq \varphi^{-1}(U)$. Since $\varphi= f\circ \psi$, we have $f(V)=f(\psi(\psi^{-1}(V))=\varphi(\psi^{-1}(V))=\varphi(W)\subseteq U$. Hence, $f$ is pre-continuous at the neutral element of $K$. Therefore $f$ is pre-continuous.
\end{proof}

\begin{corollary}\label{f4}
Let $\varphi: G\rightarrow H$ and $\psi: G\rightarrow K $ be a pre-continuous homomorphism of semi-pre-topological groups $G$, $H$ and $K$ such that $\psi(G)= K$ and $ker \psi \subseteq ker\varphi$. If the homomorphism $\psi$ is open, then there exists a pre-continuous homomorphism $f: K\rightarrow H$ such that $\varphi= f\circ \psi$.
\end{corollary}

\begin{proof}
 It follows from proposition~\ref{g4} that there exists a homomorphism $f: K\rightarrow H$ such that $\varphi= f\circ \psi$. Then it suffice to prove that $f$ is pre-continuous. Let $V$ be an arbitrary open set in $H$, then $f^{-1}(V)= \psi(\varphi^{-1}(V))$.  Since $\varphi$ is pre-continuous and $\psi$ is open, we conclude that $f^{-1}(V)$ is open in $K$. Thus $f$ is pre-continuous.
\end{proof}

\begin{proposition}
Let $G$ and $H$ be pre-topological groups and $p$ be a pre-topological isomorphism of $G$ onto $H$. If $G_0$ is a closed invariant subgroup of $G$ and $H_0= p(G_0)$, then the quotient groups $G/G_0$ and $H/H_0$ are pre-topologically isomorphic. The corresponding pre-isomorphism $\phi: G/G_0\rightarrow H/H_0$ is given by the formula $\phi(xG_0)=yH_0$, where $x\in G$ and $y= p(x)$.
\end{proposition}

\begin{proof}
Let $\varphi: G\rightarrow G/G_0$ and $\psi: H\rightarrow H/H_0$ be the pre-quotient homomorphisms. We can easily prove that $\phi$ is a homomorphism of $G/G_0$ onto $H/H_0$. From the definition of $\phi$, it follows that $\psi\circ p= \phi\circ \varphi$. Since $p$, $\varphi$ and $\psi$ are open pre-continuous homomorphisms, so is $\phi$. Take an arbitrary element $xG_0$ of $G/G_0$ and set $y= p(x)$. If $\pi(xG_0)= H_0$, then $\psi(y)= H_0$, where $y\in H_0$ and $x\in G_0$, hence ker of $\phi$ is trivial. In other words, $\phi$ is an pre-isomorphism. Thus, $\phi$ is pre-topological isomorphism.
\end{proof}

\begin{theorem}
(First Isomorphism)

Let $G$ and $H$ be semi-pre-topological groups with neutral elements $e_{G}$ and $e_{H}$, respectively, and let $p$ be an open pre-continuous homomorphism of $G$ onto $H$. Then kernel $N=p^{-1}(e_{H})$ of $p$ is a closed invariant subgroup of $G$, and the fibers $p^{-1}(y)$ with $y \in H$ coincide with the cosets of $N$ in $G$. The mapping $\Phi:G/N\rightarrow H$ which assigns to a coset $xN$ the element $p(x)\in H$ is a pre-topological isomorphism.
\end{theorem}

\begin{theorem}
Let $G$ be an almost topological group, $H$ a closed subgroup of $G$, and $\pi:G\rightarrow G/H$ be the canonical mapping. If $K$ is a dense subgroup of $G$, then the restriction $r=\pi\upharpoonright K$ is an open mapping of $K$ onto $\pi (K)$.
\end{theorem}

\begin{proof}
Take any non-empty open set $U$ in $K$. Then there an open set $V$ in $G$ such that $U= K\cap V$. Clearly, $r(U)=\pi(K\cap V)\subseteq \pi (K)\cap \pi(V)=O$. Since the mapping $\pi $ is open, the set $O=\pi(V) \cap \pi (K)$ is open in $\pi(K)$, then we have $r(U)=O$. Indeed, for any $y\in O$, there exists $x\in K$ such that $\pi(x)=y$, which implies that $xH \cap V=\pi^{-1}(y) \cap V \neq \emptyset$. Since $K\cap H$ is dense in $H$, it follows that $x(K\cap H)=K\cap xH$ is dense in $xH$. Thus $(K\cap xH)\cap V \neq \emptyset$, so there is a point $x^{\prime} \in K \cap xH \cap V= U \cap xH$. Hence $r(x^{\prime})= \pi (x^{\prime})= \pi (x)= y$, that is, $y\in r(U)$; then it follows that $r(U)=O$.  Therefore, the mapping $r: K \rightarrow \pi(K)$ is open.
\end{proof}

\begin{question}
Let $G$ be an almost topological group, $H$ a closed subgroup of $G$, and $\pi:G\rightarrow G/H$ be the canonical mapping. If $K$ is a dense subgroup of $G$ and the restriction $r=\pi\upharpoonright K$ is an open mapping of $K$ onto $\pi (K)$, is the intersection $K\cap H$ dense in $H$?
\end{question}

\begin{theorem}
Let $p:G\rightarrow H$ be pre-continuous homomorphism of almost topological groups. Suppose that the image $p(U)$ contains a non-empty open set in $H$, for each open neighbourhood $U$ of the neutral element $e_{G}$ in $G$. Then the homomorphism $p$ is open.
\end{theorem}

\begin{proof}
First we claim that for each open neighbourhood $U$ of the neutral element $e_{G}$ in $G$, there is an open neighbourhood $W$ of neutral element $e_{H}$ of $H$ is contained in  $p(U)$. Indeed, since $G$ is an almost topological group, there exists an open neighbourhood $V$ of the neutral element $e_{G}$, such that $V^{-1}V\subseteq U$. By assumption, $p(V)$ contains a non-empty open set $W$ in $H$, then $W^{-1}W$ is an open neighbourhood of $e_{H}$ and $W^{-1}W\subseteq p(V)^{-1}p(V)= p(V^{-1}V)\subseteq p(U)$.

Now we prove this theorem. Take any open set $U$ in $G$ and any element $y \in p(U)$; there is a point $x\in U$ such that $y=p(x)$. So, we can find an open neighbourhood $O$ of the neutral element $e_{G}$ such that $xO\subseteq U$. According to our claim above, there exists an open neighbourhood $W_{1}$ of neutral element $e_{H}$ such that
$W_{1}\subseteq p(O)$, hence $$y\in yW_{1}\subseteq yp(O)\subseteq p(x)p(O) \subseteq p(xO)\subseteq p(U).$$
By the arbitrary choice of $y$, it follows that $p(U)$ is open in $H$.
\end{proof}

The next two results are known as the second isomorphism theorem and the third isomorphism theorem. We give them without any proofs.

\begin{theorem}
(Second Isomorphism)
Let $G$ and $H$ be left pre-topological groups with the neutral elements $e_{G}$ and $e_{H}$, respectively, and let$p: G\rightarrow H$ be an open pre-continuous homomorphism of $G$ onto $H$. Let $H_{0}$ be closed invariant subgroup of $H$, $G_{0}= p^{-1}(H_{0})$ and $N= p^{-1}(e_{H})$. Then the left pre-topological groups $G/G_{0}$, $ H/H_{0}$, and $(G/N)/(G_{0}/N)$ are pre-topologically isomorphic.
\end{theorem}

\begin{theorem}
(Third Isomorphism)
Let $G$ be a pre-topological group, $H$ be a closed invariant subgroup of $G$ and $M$ be any pre-topological subgroup of $G$. Then the quotient group $MH/H$ is pre-topologically isomorphic to the subgroup $\pi(M)$ of the pre-topological group $G/H$, where $\pi: G\rightarrow G/H$ is the natural quotient homomorphism.
\end{theorem}

\smallskip
\section{The complete regularity and the character of pre-topological groups}
In this section, we mainly discuss the complete regularity of pre-topological groups, and give a characterization of almost topological groups with the character less than or equal to $\tau$. Since each $T_{0}$ topological group is completely regular, it is natural to pose the following question.

\begin{question}\label{qq}
Is each $T_{0}$ pre-topological group completely regular?
\end{question}

In this section we shall prove that each almost topological group is completely regular, which gives a partial answer to Question~\ref{qq}. First, we give some lemmas.

The proofs of these lemmas are similar to the proofs of \cite[Lemmas 3.3.7, 3.3.8 and 3.3.10]{AT}, thus we omit them.

Assume $N$ is a prenorm on a group $G$. Put $B_{N}(\varepsilon)=\{x\in G: N(x)<\varepsilon\}$ for each $\varepsilon>0$, which is called the $N$-ball of radius $\varepsilon$. Obviously, the ball $B_{N}(\varepsilon)$ is an open set of $G$ if $N$ is a pre-continuous prenorm.

\begin{lemma}\label{lll0}
A prenorm $N$ on a pre-topological group $G$ is pre-continuous if and only if for every $\varepsilon>0$ there is an open neighborhood $U$ of the neutral element $e$ such that $U\subseteq B_{N}(\varepsilon)$.
\end{lemma}

A semi-pre-topological group $G$ is called {\it left pre-uniformly Tychonoff} (resp. {\it right pre-uniformly Tychonoff}) if for each open neighborhood $U$ of the neutral element $e$, there exists a left (resp. right) pre-uniformly continuous function $f$ on $G$ such that $f(e)=0$ and $f(x)\geq 1$ for each $x\in G\setminus U$. Clearly, each left (or right) pre-uniformly Tychonoff is completely regular.

\begin{lemma}\label{lll2}
Each pre-continuous prenorm on a pre-topological group $G$ is a pre-uniformly continuous function with respect to left and right group pre-uniformities on $G$.
\end{lemma}

\begin{lemma}\label{lll1}
Let $G$ be an almost topological group, and let $\{U_{n}: n\in\omega\}$ be a sequence of symmetric open neighborhoods of the neutral element $e$ such that $U_{n+1}\cdot U_{n+1}\subset U_{n}$ for each $n\in \omega$. Then there exists a prenorm $N$ on $G$ satisfies the following conditions:
$$\{x\in G: N(x)<1/2^{n}\}\subset U_{n}\subset\{x\in G: N(x)\leq 2/2^{n}\}\ \mbox{for any}\ n\in\omega.$$
Hence, this prenorm $N$ is pre-continuous. Moreover, if each set $U_{n}$ are invariant, then $N$ on $G$ can be chosen to satisfy $N(xyx^{-1})=N(y)$ for any $x, y\in G$.
\end{lemma}

By Lemmas~\ref{lll0} and~\ref{lll1}, we have the following theorem, which is a generalization of A.A. Markov's theorem.

\begin{theorem}\label{t11}
For each open neighborhood $U$ of the neutral element $e$ of an almost topological group $G$, there is a pre-continuous prenorm $N$ on $G$ such that the unit ball $B_{N}(1)\subseteq U$.
\end{theorem}

Now we can prove one of main results in this section.

\begin{theorem}\label{t20227}
Every almost topological group $G$ is completely regular.
\end{theorem}

\begin{proof}
Take any open neighborhood $U$ of the neutral element $e$ in $G$. From Theorem~\ref{t11}, there exists a pre-continuous prenorm $N$ on $G$ such that $B_{N}(1)\subseteq U$. Therefore, we have $N(e)=0$ and $N(x)\geq 1$ for each $x\in G\setminus U$. Since $N$ is pre-continuous, it follows from Lemma~\ref{lll2} that $G$ is completely regular.
\end{proof}

The following theorem gives a characterization of almost topological groups with the character less than or equal to $\tau$.

\begin{theorem}\label{t20221}
Let $G$ be an almost topological group. Then $\chi(G)\leq\tau$ if and only if there exists a family $\{\rho_{\alpha}: \alpha<\tau\}$ of right-invariant pseudometrics such that the family $$\{B_{\rho_{\alpha}}(\frac{1}{2^{n}}): \alpha<\tau, n\in\mathbb{N}\}$$ is a pre-base at the neutral element $e$.
\end{theorem}

\begin{proof}
The sufficiency is obvious. We only need to prove the necessity. Fix a pre-base $\{U_{\alpha}: \alpha<\tau\}$ of $G$ at the neutral element $e$, where we may assume that each $U_{\alpha}$ is symmetric since $G$ is an almost topological group. Fix any $\alpha<\tau$; since $G$ is an almost topological group, we can take a subsequence $\{U_{\alpha_{n}}: n\in\mathbb{N}\}$ of $\{U_{\alpha}: \alpha<\tau\}$ such that $U_{\alpha_{1}}=U_{\alpha}$ and $U_{\alpha_{n+1}}^{2}\subseteq U_{\alpha_{n}}$ for each $n\in\mathbb{N}$. By Lemma~\ref{lll1}, there exists a pre-continuous prenorm $N_{\alpha}$ on $G$ such that $B_{N_{\alpha}}(\frac{1}{2^{n}})\subseteq U_{\alpha_{n}}$ for each $n\in\mathbb{N}$. Moreover, each $B_{N_{\alpha}}(\frac{1}{2^{n}})$ is open in $G$. Now, for any $x$ and $y$ in $G$, put $\rho_{\alpha}=N_{\alpha}(xy^{-1})$. Then it is easy to check that $\rho_{\alpha}$ is a right-invariant pseudometric. For each $n\in\mathbb{N}$, since $B_{\rho_{\alpha}}(\frac{1}{2^{n}})=B_{N_{\alpha}}(\frac{1}{2^{n}})$, it follows that $B_{\rho_{\alpha}}(\frac{1}{2^{n}})$ is open in $G$.

Finally, we prove that the family $\{B_{\rho_{\alpha}}(\frac{1}{2^{n}}): \alpha<\tau, n\in\mathbb{N}\}$ is a pre-base at the neutral element $e$. Indeed, take any open neighborhood $O$ of $e$; then there exists $\alpha<\tau$ such that $U_{\alpha}\subseteq O$. Hence $B_{\rho_{\alpha}}(\frac{1}{2})\subseteq U_{\alpha}\subseteq O$. The proof is completed.
\end{proof}

One can complement Theorem~\ref{t20221} as follows:

\begin{corollary}
Let $G$ be an almost topological group. Then $\chi(G)\leq\tau$ if and only if there exist a family $\{\rho_{\alpha}: \alpha<\tau\}$ of right-invariant pseudometrics and a family $\{\sigma_{\alpha}: \alpha<\tau\}$ of left-invariant pseudometrics, both generating the original pre-topology of $G$.
\end{corollary}

\begin{proof}
As in the proof of Theorem~\ref{t20221}, for each $\alpha<\tau$ take a pre-continuous prenorm $N_{\alpha}$, and put $\rho_{\alpha}=N_{\alpha}(xy^{-1})$ and $\sigma_{\alpha}=N_{\alpha}(x^{-1}y)$ for any $x, y\in G$; then $\rho_{\alpha}$ and $\sigma_{\alpha}$ are right-invariant and left-invariant pseudometrics on $G$ respectively. As is was shown in Theorem~\ref{t20221}, both the families $\{\rho_{\alpha}: \alpha<\tau\}$ and $\{\sigma_{\alpha}: \alpha<\tau\}$ generate the original pre-topology of $G$ respectively.
\end{proof}

The following corollary generalizes the well known Birkhoff and Kakutani's Theorem.

\begin{corollary}\label{c2022}
An almost topological group $G$ is first-countable if and only if there exists a sequence $\{\rho_{n}: n\in\mathbb{N}\}$ of right-invariant pseudometrics such that the family $$\{B_{\rho_{n}}(\frac{1}{2^{m}}):  n, m\in\mathbb{N}\}$$ is a pre-base at the neutral element $e$.
\end{corollary}

\begin{corollary}
An almost topological group $G$ admits a family of invariant pseudometrics generating its pre-topology if and only if $G$ is balanced.
\end{corollary}

\begin{proof}
Necessity. Suppose that $G$ is generated by a family $\{\rho_{\alpha}: \alpha\in I\}$ of invariant pseudometrics. For each $\alpha\in I$ and $n\in\mathbb{N}$, denote by $U_{\alpha, n}$ the $\frac{1}{n}$-ball center at the neutral element $e$ of $G$ with respect to $\rho_{\alpha}$. Now it suffices to prove that $x U_{\alpha, n}x^{-1}=U_{\alpha, n}$ for any $x\in G$, $\alpha\in I$ and $n\in\mathbb{N}$. Fix any $x\in G$, $\alpha\in I$ and $n\in\mathbb{N}$, since $$\rho_{\alpha}(e, xyx^{-1})=\rho_{\alpha}(x, xy)=\rho_{\alpha}(e, y)<\frac{1}{n},$$it follows that $x U_{\alpha, n}x^{-1}=U_{\alpha, n}$. Hence the pre-topological group $G$ is balanced.

Sufficiency. Suppose that $G$ is balanced, and that $\chi(G)\leq\tau$. Then there exists a family $\mathscr{N}=\{U_{\alpha}: \alpha<\tau\}$ of open, symmetric, invariant neighborhoods of $e$ in $G$ such that $\mathscr{N}$ forms a pre-base for $G$ at $e$. Hence for each $\alpha<\tau$ it follows from Lemma~\ref{lll1} that there exists a pre-continuous prenorm $N_{\alpha}$ on $G$ such that $$\{x\in G: N_{\alpha}(x)<1/2\}\subset U_{\alpha}\subset\{x\in G: N_{\alpha}(x)\leq 1\}$$ and $N_{\alpha}(xyx^{-1})=N(y)$ for any $x, y\in G$; then the pseudometric $\rho_{\alpha}$ defined by $$\rho_{\alpha}=N_{\alpha}(x^{-1}y)=N_{\alpha}(xy^{-1})$$ is invariant. It is easily checked that the family $\{\rho_{\alpha}: \alpha\in I\}$ generates the pre-topology of $G$.
\end{proof}

Let $X$ be a pre-topological space. If the pre-topology of $X$ is generated by a family $\{\rho_{\alpha}: \alpha<\tau\}$ of pseudometrics, then we say that $X$ is {\it $\tau$-metrizable}.

\begin{theorem}
Let $H$ be a closed subgroup of a $\tau$-metrizable pre-topological group $G$. Then the quotient pre-topological space $G/H$ is $\tau$-metrizable.
\end{theorem}

\begin{proof}
 From Corollary~\ref{c2022}, there exist a family $\{\rho_{\alpha}: \alpha<\tau\}$ of right-invariant pseudometrics generating the original pre-topology of $G$. Fix any $\alpha<\tau$. For arbitrary points $x, y\in G$, let $d_{\alpha}(xH, yH)$ be the number $$d_{\alpha}(xH, yH)=\inf\{\rho_{\alpha}(xh_{1}, yh_{2}): h_{1}, h_{2}\in H\}.$$By a similar proof of \cite[Proposition 3.3.19]{AT}, it is easily checked that $d_{\alpha}$ is a pseudometric.

To finish the proof, we need to prove that the family $\{d_{\alpha}: \alpha<\tau\}$ generates the pre-topology of the quotient pre-topological space $G/H$. Indeed, denote by $\pi$ be the quotient mapping of $G$ onto $G/H$, $\pi(x)=xH$ for each $x\in G$. For any $x\in G$, $\alpha<\tau$ and $\varepsilon>0$, put $$O_{\alpha, \varepsilon}(x)=\{y\in G: \rho_{\alpha}(x, y)<\varepsilon\}\ \mbox{and}$$ $$B_{\alpha, \varepsilon}(xH)=\{yH: y\in G, d_{\alpha}(xH, yH)<\varepsilon\}.$$ For each $\alpha<\tau$, it follows from the definition of the pseudometric $\rho_{\alpha}$ that $\pi(O_{\alpha, \varepsilon}(x))=B_{\alpha, \varepsilon}(xH)$ for each $x\in G$ and $\varepsilon>0$. Since the family $\{O_{\alpha, \varepsilon}(x): x\in G, \alpha<\tau, \varepsilon>0\}$ form a pre-base for $G$ and the mapping $\pi$ is pre-continuous and open by Theorem~\ref{a4}, we claim that $\{B_{\alpha, \varepsilon}(xH): x\in G, \alpha<\tau, \varepsilon>0\}$ is a pre-base for the original pre-topology of the pre-topological space $G/H$. Therefore, $G/H$ is $\tau$-metrizable.
\end{proof}

The following is easily checked, we left the proof to the reader.

\begin{lemma}\label{l2022}
Assume that $f: X\rightarrow Y$ is an open pre-continuous mapping of a pre-topological space $X$ onto a pre-topological space $Y$. Then $\overline{f^{-1}(B)}=f^{-1}(\overline{B})$.
\end{lemma}

\begin{theorem}
Assume that $G$ is a pre-topological group and $H$ is a closed subgroup of $G$. If $H$ and $G/H$ are separable, then $G$ is also separable.
\end{theorem}

\begin{proof}
Suppose that $\pi$ is the natural homomorphism of $G$ onto thequotient pre-topological space $G/H$. From the separability of $G/H$, we can fix a dense countable subset $A$ of $G/H$. Since $H$ is separable and each coset $xH$ is pre-homeomorphic to $H$, it follows that we can take a dense countable subset $D_{y}$ of $\pi^{-1}(y)$ for each $y\in A$. Put $D=\bigcup\{D_{y}: y\in A\}$; then $D$ is a countable subset of $G$ and $D$ is dense in $\pi^{-1}(A)$. Since $\pi$ is open, from Lemma~\ref{l2022} it follows that $\overline{\pi^{-1}(A)}=\pi^{-1}(\overline{A})=\pi^{-1}(G/H)=G$. Therefore, $G$ is separable.
\end{proof}

\begin{question}
Let $G$ be a pre-topological group and $H$ be a closed pre-topological subgroup. If the pre-topological subspace $H$ and $G/H$ are first-countable, is then the pre-topological space $G$ also first-countable?
\end{question}

\section{The index $\tau$-narrowness in pre-topological groups}

In this section, some cardinal invariants of pre-topological groups are studied. In particular, the well-known Guran's Theorem is extended, that is, an almost topological group is $\tau$-narrow if and only if it can be embedded as a subgroup of a pre-topological product of almost topological groups of weight less than or equal to $\tau$.

\begin{definition}
A semi-pre-topological group is called {\it left $\tau$-narrow} (resp. {\it right $\tau$-narrow}) if, for each open neighborhood $U$ of the neutral element in $G$, there exists a subset $F$ of $G$ such that $G= FU$ (resp. $G=UF$)) and $|F|\leq\tau$. If $G$ is left $\tau$-narrow and right $\tau$-narrow then $G$ is called $\tau$-narrow. The {\it index of narrowness} of a semi-pre-topological group $G$ denoted by $ib(G)$, that is, the minimal cardinal $\tau\geq\omega$ such that $G$ is $\tau$-narrow.
\end{definition}

First, we give some basic properties of $\tau$-narrowness of pre-topological groups.

\begin{proposition}\label{p202210}
The following conditions are equivalent for a quasi-pre-topological group $G$.

\smallskip
(1) $G$ is $\tau$-narrow;

\smallskip
(2) For every open neighbourhood $V$ of $e$ in $G$, there exists a subset $B\subseteq G$ with $|B|\leq \tau$ such that $G=VB$;

\smallskip
(3) For every open neighbourhood $V$ of $e$ in $G$, there exists a countable set $C\subseteq G$ with $|C|\leq \tau$ such that $CV=VC=G$.
\end{proposition}

\begin{proof}
Clearly,(3) $\Rightarrow$ (1). Now we only need to prove (1) $\Rightarrow$ (2) and(2) $\Rightarrow$ (3).

\smallskip
(1) $\Rightarrow$ (2). Let $G$ be a $\tau$-narrow quasi-pre-topological group. For any open neighbourhood $V$ of $e$, there is an open neighbourhood $U$ such that $U^{-1}\subseteq V$. Therefore, we can find a subset $A$ of $G$ such that $G=AU$ and $|A|\leq\tau$. Put $B= A^{-1}$; then $G=G^{-1}=(AU)^{-1}=U^{-1}A^{-1}\subseteq VB$, that is, $G=VB$.

\smallskip
(2) $\Rightarrow$ (3). For every $V\in \mathscr{B}_{e}$, there exists $U\in \mathscr{B}_{e}$ such that $U^{-1}\subseteq V$. By our assumption, there exist subsets $B$ and $A$ of $G$ such that $VA=G$ ,$G=UB$, $|A|\leq\tau$ and $|B|\leq\tau$, then it follows from $G=G^{-1}=B^{-1}U^{-1}\subseteq B^{-1}V$ that $G=B^{-1}V$. Put $C=A\cup B^{-1}$; then $|C|\leq\tau$ and $CV=G=VC$.
\end{proof}

\begin{proposition}\label{p20223}
If pre-topological group $H$ is a pre-continuous homomorphic image of a $\tau$-narrow pre-topological group $G$, then $H$ is also $\tau$-narrow.
\end{proposition}

\begin{proof}
Let $V$ be an open neighbourhood of neutral element $e$ in $H$ and $f:G\rightarrow H$ be a pre-continuous homomorphic mapping. Since $G$ is $\tau$-narrow, there exists a subset $A$ of $G$ such that $Af^{-1}(V)=G$ and $A|\leq\tau$. It follows from $f$ is homomorphic that $f(G)=f(Af^{-1}(V))=f(A)V=H$. Clearly, $f(A)$ is a subset of $H$ and $|f(A)|\leq\tau$, hence $H$ is $\tau$-narrow.
\end{proof}

The following proposition is obvious, so we leave the proof to the reader.

\begin{proposition}\label{p20225}
The pre-topological product of an arbitrary family of $\tau$-narrow pre-topological group is a $\tau$-narrow pre-topological group.
\end{proposition}

It is well-known that each subgroup of a $\tau$-narrow topological group is $\tau$-narrow. Hence we have the following question.

\begin{question}\label{a5}
When is a subgroup $H$ of a $\tau$-narrow pre-topological group $G$ $\tau$-narrow?
\end{question}

Indeed, the situation are different in the class of pre-topological groups, see the following two examples.

\begin{example}
There exists a closed subgroup $H$ of an $\omega$-narrow strongly pre-topological group $G$ such that $H$ is not $\omega$-narrow.
\end{example}

\begin{proof}
Let $G$ be the group $(R^{2}, +)$ with usual addition which is endowed with a pre-topology such that the following family is a pre-basis $\mathscr{B}_{e}$ at the neutral element $(0, 0)$:
$$\mathscr{B}_{e}=\{(-\frac{1}{n},0]\times (-\frac{1}{n},0]: n\in \mathbb{N}\}\cup \{[0, \frac{1}{n})\times [0, \frac{1}{n}):n\in \mathbb{N}\}.$$ Then $G$ is an $\omega$-narrow strongly pre-topological group.
Let $H=\{(x, y):x+y=0\}$. Then $H$ is a closed subgroup $H$ of $G$. However, $H$ is a discrete topological group, hence $H$ is not $\omega$-narrow.
\end{proof}

\begin{example}
There exists an open subgroup $H$ of an $\omega$-narrow pre-topological group $G$ is not $\omega$-narrow.
\end{example}

\begin{proof}
Let $H$ be the group $(R^{2}, +)$ with usual addition and $e=(0, 0)$. Put $$\mathcal{U}=\{[0, \frac{1}{n})\times \{0\}, (-\frac{1}{n}, 0]\times \{0\}: n\in\mathbb{N}\}\cup\{\{0\}\times[0, \frac{1}{n}), \{0\}\times(-\frac{1}{n}, 0]: n\in\mathbb{N}\}.$$ Then $\mathcal{U}$ is a family of subsets of $H$ satisfying conditions (1)-(4) of Theorem~\ref{t0}. Then it follows from Theorem~\ref{t2022} that the family $\mathcal{B}_{\mathscr{U}}=\{Ua:
a\in H, U\in\mathscr{U}\}$ is a pre-base for a pre-topology $\tau$ on $H$ such that $(H, \tau)$ is a pre-topological group. Clearly, $H$ is $\omega$-narrow and subgroup $\mathbb{R}\times\{0\}$ is open in $H$. However, $\mathbb{R}\times\{0\}$ is not $\omega$-narrow since $e$ is open in the uncountable subgroup $\mathbb{R}\times\{0\}$.
\end{proof}

The next two theorems give partial answers to question~\ref{a5}.

\begin{theorem}\label{t202210}
Each subgroup $H$ of a $\tau$-narrow almost topological group $G$ is $\tau$-narrow.
\end{theorem}

\begin{proof}
Let $W$ be an open neighbourhood of the neutral element $e$ in $H$. Then there exists an open symmetric neighbourhood $V$ of $e$ in $G$ such that $V^{2}\cap H\subseteq W$. Since $G$ is $\tau$-narrow, there exists a subset $B$ of $G$ such that $BV=G$ and $|B|\leq\tau$. Let $$C=\{c\in B: cV\cap H\neq \varnothing\}.$$ Then $|C|\leq\tau$ and $H\subseteq CV$. For each $c\in C$, choose an element $a_{c}\in cV\cap H$, then put $A=\{a_{c}: a\in C\}$. Since $|C|\leq\tau$, it follows that is a subset of $H$ with $|A|\leq\tau$. We conclude that $AW=H$. Indeed, since $H$ is a subgroup of $G$ and $V^{2}\cap H\subseteq W\subseteq H$, we conclude that $(AV^{2})\cap H\subseteq AW\subseteq H$. Obviously, $A\subseteq H\subseteq CV$. Since $V$ is symmetric, we have $C\subseteq AV$, hence $H\subseteq CV\subseteq AV^{2}\subseteq AW$. Therefore, $H$ is $\tau$-narrow.
\end{proof}

\begin{theorem}\label{t202299}
Every dense subgroup $H$ of a $\tau$-narrow pre-topological group $G$ is $\tau$-narrow.
\end{theorem}

\begin{proof}
Let $W$ be an open neighbourhood of the identity $e$ in $H$. Then there exist open neighbourhoods $V_{1}, V_{2}$ of $e$ in $G$ such that $V_{1}V_{2}\cap H\subseteq W$. Since $G$ is $\tau$-narrow, there exists a subset $C$ with $|C|\leq\tau$ such that $G=CV_{1}=CV_{2}=CV_{1}^{-1}$. For each $c\in C$, we have that $cV_{1}^{-1}\cap H \neq \varnothing$, then fix an element $a_{c}\in cV_{1}^{-1}\cap H$. Put $A=\{a_{c}:a\in C\}$; then $|A|\leq\tau$. We claim that $AW=H$. Indeed, from our definition of $A$ it follows that $C\subseteq AV_{1}$. Then $H\subseteq CV_{2}\subseteq AV_{1}V_{2}$, hence $H\subseteq A(V_{1}V_{2}\cap H)\subseteq AW$. Thus $AW=H$.
\end{proof}

We say that a pairwise disjoint family consisting of non-empty open subsets of a pre-topological space $(Z, \tau)$ is called a {\it cellular family}. The {\it cellularity} of $Z$ is defined as follows: $$c(Z)=\sup\{|\mathscr{V}|: \mathscr{V}\ \mbox{ is a cellular family in}\ Z\}.$$ Here, the cellularity of a pre-topological space maybe finite.

It is well-known that $ib(G)\leq c(G)$ for any topological group $G$. Therefore, it is natural to pose the following question.

\begin{question}\label{q2022888}
Let $G$ be a pre-topological group (or strong pre-topology). Does $ib(G)\leq c(G)$ hold?
\end{question}

Next we give some partial answers to Question~\ref{q2022888}.

\begin{definition}
Let $X$ be a pre-topological space. The smallest number $\kappa$ such that each open cover $\mathscr{U}$ of $X$ has a subfamily $\mathscr{V}$ of $\mathscr{U}$ with $|\mathscr{V}|\leq\kappa$ and $\bigcup\mathscr{V}=X$ is called the {\it Lindel\"{o}f number} of the pre-topological space $X$ and is denoted by $l(X)$. If $l(X)=\omega$, then $X$ is called {\it Lindel\"{o}f}.
\end{definition}

\begin{proposition}\label{p2022222}
If $G$ is a pre-topological group, then $ib(G\leq l(G)$.
\end{proposition}

\begin{proof}
Let $l(G)=\tau$. Take an arbitrary open neighborhood $U$ of the neutral element of $e$. Then $\{xU: x\in G\}$ is an open cover of $G$. Since $l(G)\leq\tau$, there exists a  subset $A$ such that $|A|\leq\tau$ and $\{xU: x\in A\}$ is a cover of $G$, that is, $G=\bigcup_{x\in A}xU=AU$. Hence, $G$ is $\tau$-narrow.
\end{proof}

\begin{corollary}\label{c20222}
If $G$ is a Lindel\"{o}f pre-topological group, then $G$ is $\omega$-narrow.
\end{corollary}

By Proposition~\ref{u0} or Corollary~\ref{c20222}, each separable (left) pre-topological group is $\omega$-narrow.

\begin{theorem}\label{t202211}
If $G$ is an almost topological group, then $ib(G)\leq c(G)$.
\end{theorem}

\begin{proof}
Let $c(G)\leq\tau$. We claim that $ib(G)\leq\tau$. Indeed, pick an arbitrary open neighbourhood $U$ of the neutral element $e$ of $G$. Then there exists a symmetric and open neighbourhood $V$ of $e$ such that $V^{2}\subseteq U$. Since the family $\zeta$ of all $V$-disjoint subsets of $G$ is (partially) order by inclusion, and the union of any chain of $V$-disjoint sets is also a $V$-disjoint set. By the Zorn's Lemma, we can find a maximal element $A$ of the ordered set $\zeta$. Obviously, $\{aV: a\in A\}$ is a disjoint family of non-empty open sets in $G$. Since $c(G)\leq\tau$, the set $|A|\leq\tau$.
From the maximality of $A$, it follows that for every $x\in G\setminus A$ there exists $a\in A$ such that $xV\cap aV\neq \varnothing$, then $x\in aVV^{-1}=aV^{2}\subseteq aU$. Hence $AU=G$. Therefore, $ib(G)\leq\tau$.
\end{proof}

Given a pre-topological space $X$, we denote by $e(X)$ the supremum of cardinalities of closed discrete subsets of $X$. The cardinal invariant $e(X)$ is called the {\it extent} of $X$. Obviously, each Lindel\"{o}f pre-topological space is countable. It is well-known that for any topological group $G$, we have $ib(G)\leq e(G)$. However, the following question is still un-known for us in the class of almost topological groups.

\begin{question}\label{q20221}
If $G$ is an almost topological group, then does $ib(G)\leq e(G)$ hold?
\end{question}

The following theorem gives a complement for Theorem~\ref{t202299}.

\begin{theorem}\label{t202213}
If  a pre-topological group $G$ contains a dense subgroup $H$ such that $H$ is $\tau$-narrow, then $G$ is also $\tau$-narrow.
\end{theorem}

\begin{proof}
Let $U$ be any open neighbourhood of the neutral element $e$ in $G$; then there exist open neighbourhoods $V_{1}, V_{2}$ of the neutral element $e$ of $G$ such that $V_{1}V_{2}\subseteq U$. Since $H$ is $\tau$-narrow, there is a subset $A$ of $H$ such that $A\leq\tau$ and $H\subseteq AV_{1}$. By Proposition~\ref{u0}, $G=\overline{H}\subseteq \overline{AV_{1}}\subseteq AV_{1}V_{2}\subseteq AU$, thus $AU=G$. Hence $G$ is $\tau$-narrow.
\end{proof}

The following proposition gives a relation of the weight of an almost topological group $G$ and the narrowness of $G$.

\begin{proposition}\label{p20224}
Let $G$ be an almost topological group. Then $\omega(G)=ib(G)\chi(G)$.
\end{proposition}

\begin{proof}
Clearly, $\chi(G)\leq w(G)$. By Proposition~\ref{p2022222}, we have $ib(G)\leq l(G)\leq w(G)$. Hence $ib(G)\chi(G)\leq\omega(G)$. We only need to prove $\omega(G)\leq ib(G)\chi(G)$. Let $ib(G)\leq\tau$ and $\chi(G)\leq\kappa$. Then we assume that $\{U_{\alpha}: \alpha\in \kappa\}$ is a pre-base at the identity $e$ of $G$. For every $\alpha\in \kappa$, there exists a subset $C_{\alpha}$ of $G$ with $|C_{\alpha}|\leq\tau$ such that $C_{\alpha}U_{\alpha}=G$. Then the cardinality of the family $\mathscr{B}=\{xU_{\alpha}:x\in C_{\alpha}, \alpha\in \kappa\}$ is at most $\tau\kappa$. We claim that $\mathscr{B}$ is a pre-base of $G$.

Indeed, it suffices to prove that for any neighbourhood $O$ of an arbitrary point $a\in G$ there exist $\alpha\in\kappa$ and $x\in C_{\alpha}$ such that $a\in xU_{\alpha}\subseteq O$. Take an any neighbourhood $O$ of an arbitrary point $a\in G$. Since $G$ is an almost topological group, it follows that there are $\alpha, \beta\in \kappa$ such that $aU_{\beta}\subseteq O$ and $U^{-1}_{\alpha}U_{\alpha}\subseteq U_{\beta}$. Since $C_{\alpha}U_{\alpha}=G$, there exists $x\in C_{\alpha}$ such that $a\in xU_{\alpha}$, that is, $x\in aU^{-1}_{\alpha}$, then we have
$$xU_{\alpha}\subseteq (aU^{-1}_{\alpha})U_{\alpha}=a(U^{-1}_{\alpha}U_{\alpha})\subseteq aU_{\beta}\subseteq O.$$
Then $xU_{\alpha}$ is an open neighbourhood of point $a$ and $xU_{\alpha}\subseteq O$.
\end{proof}

\begin{remark}
The strongly pre-topological group $G$ in (3) of Example~\ref{eee} is $\omega$-narrow and $\chi(G)\leq\omega$. However, it is easy to see that $\omega(G)>\omega$. Moreover, it is natural to consider the following question.
\end{remark}

\begin{question}
Let $G$ be a symmetrically pre-
topological group. Does $\omega(G)\leq ib(G)\chi(G)$ hold?
\end{question}

 Let $G$ be a pre-topological group. We say that the {\it invariance number} of $G$ is less than or equal to $\tau$ or in symbols, inv($G$) if for any open neighbourhood $U$ of the neutral element $e$ of $G$, there exists a family $\gamma$ of open neighbourhoods of $e$ with $|\gamma|\leq\tau$ such that for each $x\in G$ there exists $V\in \gamma$ satisfying $xVx^{-1}\subseteq U$. Any such family $\gamma$ will be called {\it subordinated} to $U$. If a pre-topological group $G$ satisfy that $inv(G)\leq \tau$ , then $G$ is called {\it $\tau$-balanced}.

 It is well known that each $\tau$-narrow topological group is $\tau$-balanced. Moreover, it is obvious that each Abelian pre-topological group is $\tau$-balanced. Hence it is natural to pose the following question.

 \begin{question}\label{q2022}
If $G$ is a $\tau$-narrow pre-topological group, then is $G$ $\tau$-balanced?
\end{question}

The following proposition gives a partial answer to Question~\ref{q2022}.

\begin{proposition}\label{ppp1}
If $G$ is a $\tau$-narrow almost topological group, then $G$ is $\tau$-balanced.
\end{proposition}

\begin{proof}
Let $U$ be an open neighbourhood of the neutral element $e$ in $G$; then since $G$ is an almost topological group, there exists an open and symmetric neighbourhood $V$ of the neutral element $e$ such that $V^{3}\subseteq U$. Since $G$ is $\tau$-narrow, there exists a subset $A$ of $G$ with $|A|\leq\tau$ such that $VA=G$. For any $a\in A$, there exists an open neighbourhood $W_{a}$ of the neutral element $e$ such that $aW_{a}a^{-1}\subseteq V$. Then we conclude that the family  $\gamma=\{W_{a}: a\in A\}$ subordinated to $U$.

Indeed, it is obvious that $\gamma$ is a family of open neighbourhoods of $e$ and $|\gamma|\leq\tau$. For any $x\in G$, there exists $a\in A$ such that $x\in Va$. Therefore, $$xW_{a}x^{-1}\subseteq VaW_{a}a^{-1}V^{-1}\subseteq VVV^{-1}=V^{3}\subseteq U$$ Hence $\gamma$ is subordinated to $U$, that is, $G$ is $\tau$-balanced.
\end{proof}

\begin{theorem}
Let $G$ be a pre-semitopological group. Then $inv(G)\leq\chi(G)$.
\end{theorem}

\begin{proof}
Let $\chi(G)=\kappa$ and $\{U_{\alpha}: \alpha<\kappa\}$ be a pre-base of $G$ at the neutral element $e$. Take any open neighborhood $V$ of $e$. Then for each $x\in G$ the set $Vx$ is an open neighborhood of $x$. Since $G$ is a pre-semitopological group, there exists $\alpha<\kappa$ such that $xU_{\alpha}\subseteq Vx$, hence it follows that $xU_{\alpha}x^{-1}\subseteq V$. Therefore, $inv(G)\leq\kappa$.
\end{proof}

We give some properties of the invariance number of pre-topological groups.

\begin{proposition}
Each subgroup of a $\tau$-balanced pre-topological group is $\tau$-balanced.
\end{proposition}

\begin{proof}
Let $G$ be $\tau$-balanced. Take any subgroup $H$ of $G$. We claim that $H$ is also $\tau$-balanced. Indeed, take any open neighborhood $U$ of $e$ in $H$. Then there exists an open neighborhood $V$ of $e$ in $G$ such that $U=V\cap H$. Since $G$ is $\tau$-balanced, we can find a family $\eta$ of open neighborhoods of $e$ in $G$ such that $|\eta|\leq\tau$ and for each $x\in G$ there exists $W_{x}\in\eta$ satisfying $xW_{x}x^{-1}\subseteq V$. Put $\eta_{H}=\{W\cap H: W\in\eta\}$. Then, for each $x\in H$, we have $x(W_{x}\cap H)x^{-1}\subseteq (xW_{x}x^{-1})\cap H\subseteq V\cap H=U.$ Moreover, it is obvious that $|\eta_{H}|\leq\tau$. Hence $H$ is $\tau$-balanced.
\end{proof}

A subset $D$ is said to be {\it locally dense} in a pre-topological space $G$ if $\overline{U\cap D}=\overline{U}$ for each open set $U$ in $G$. Clearly, each locally dense subset of a pre-topological space is dense.

\begin{proposition}
Let $G$ be an almost topological group. If $H$ is locally dense and $\tau$-balanced subgroup of $G$, then $G$ is again a $\tau$-balanced subgroup.
\end{proposition}

\begin{proof}
Take any open neighborhood $U$ of $e$ in $G$. Then there exists an open neighborhood $V$ of $e$ in $G$ such that $V^{2}\subseteq U$. Clearly, $V\cap H$ is an open neighborhood of $e$ in $H$, hence we can find a family $\eta$ of open symmetric neighborhoods of $e$ in $H$ such that $|\eta|\leq\tau$ and for each $x\in H$ there exists $W_{x}\in\eta$ satisfying $xW_{x}x^{-1}\subseteq V\cap H.$ For each $x\in H$, there exist an open symmetric neighborhoods $U_{x}$, $V_{x}$ and $O_{x}$ of $e$ in $G$ such that $U_{x}\cap H= W_{x}$, $V_{x}^{2}\subseteq U_{x}$ and $O_{x}^{3}\subseteq V_{x}$. Put $\varphi=\{O_{x}: x\in H\}$. Clearly, $|\varphi|\leq\tau$. We claim that $\varphi$ is subordinated to $U$. Indeed, for each $y\in G$ there exists $x\in H$ such that $y\in xO_{x}$ since $G=H(\bigcap_{W\in\mathscr{B}_{e}}W)$, where $\mathscr{B}_{e}$ is the family of all open neighborhoods of $e$ in $G$. From the local density of $H$ and Proposition~\ref{u0}, it follows that $$yO_{x}y^{-1}\subseteq xO_{x}O_{x}O_{x}x^{-1}\subseteq xV_{x}x^{-1}\subseteq x\overline{V_{x}}x^{-1}=x\overline{V_{x}\cap H}x^{-1}=\overline{x(V_{x}\cap H)x^{-1}}$$$$=\overline{xW_{x}x^{-1}}\subseteq \overline{V\cap H}=\overline{V}\subseteq U.$$
\end{proof}

However, the following question is still unknown for us. If the following question is positive, then each almost topological group with a dense $\tau$-balanced subgroup is $\tau$-balanced.

\begin{question}
Is the closure of a $\tau$-balanced subgroup $H$ of an almost topological group $G$ a $\tau$-balanced subgroup?
\end{question}

The following result gives a relation of the invariance number between a pre-topological group and its co-reflexion group topology.

\begin{proposition}\label{pp20222}
Let $(G, \tau)$ be a pre-topological group. If $inv(G, \tau)\leq\kappa$, then $inv(G, \tau^{\ast})\leq\kappa$.
\end{proposition}

\begin{proof}
Take any open neighborhood $U$ of $(G, \tau^{\ast})$. From Proposition~\ref{p20221}, it follows that there exist open neighborhoods $W_{1}, \ldots, W_{m}$ in $(G, \tau)$ such that $\bigcap_{i=1}^{m}W_{i}\subseteq U$. For each $i\leq m$, since $inv(G, \tau)\leq\kappa$, there exists a family $\gamma_{i}$ of open neighbourhoods of $e$ in $(G, \tau)$ with $|\gamma_{i}|\leq\kappa$ such that $\gamma_{i}$ is subordinated to $W_{i}$. Put $\eta=\{\bigcap_{i=1}^{m}W_{i}: W_{i}\in\gamma_{i}, i\leq m\}$. From Proposition~\ref{p20221}, each element of $\eta$ is open in $(G, \tau^{\ast})$. Moreover, $|\eta|\leq\kappa$. We claim that $\eta$ is subordinated to $U$ in $(G, \tau^{\ast})$. Indeed, for any $x\in G$ and $i\leq m$, then there exists $W_{x, i}\in\gamma_{i}$ such that $x W_{x, i}x^{-1}\subseteq W_{i}$. Hence $x(\bigcap_{i=1}^{m}W_{x, i})x^{-1}\subseteq \bigcap_{i=1}^{m}W_{i}\subseteq U.$
\end{proof}

\begin{remark}
There exists an Abelian $\omega$-narrow pre-topological group such that the co-reflexion group topology is not $\omega$-narrow, such as (2) in Example~\ref{eee}.
\end{remark}

In order to give a characterization of $\tau$-narrow almost topological groups, we give some lemmas and propositions.

\begin{lemma}\label{l20222}
Let $G$ be an $\tau$-balanced almost topological group, and let $\gamma$ be a family of open neighborhoods of the neutral element $e$ in $G$ such that $|\gamma|\leq\tau$. Then there is a family $\gamma^{\ast}$ of open neighborhoods of $e$ satisfying the following properties:

\smallskip
1) $\gamma\subseteq \gamma^{\ast}$;

\smallskip
2) for each $U\in\gamma^{\ast}$, there exists a symmetric $V\in\gamma^{\ast}$ such that $V^{2}\subseteq U$;

\smallskip
3) for each $U\in\gamma^{\ast}$ and each $x\in G$, there exists $V\in\gamma^{\ast}$ such that $xVx^{-1}\subseteq U$;

\smallskip
4) $|\gamma^{\ast}|\leq\tau$.
\end{lemma}

\begin{proof}
For every $U\in\gamma$, we can find a symmetric open neighborhood $V_{U}$ of $e$ such that $V_{U}^{2}\subseteq U$; since $G$ is $\tau$-balanced, there exists a family $\mathcal{V}_{U}$ of open neighborhoods of $e$ subordinated to $U$ such that $|\mathcal{V}_{U}|\leq\tau$. Now put$$\varphi(\gamma)=\gamma\cup\{\mathcal{V}_{U}: U\in\gamma\}\cup\{V_{U}: U\in\gamma\}.$$Then we put $\gamma_{0}=\gamma$, $\gamma_{1}=\varphi(\gamma_{0})$, and repeat this operation, which defined by induction families $\gamma_{2},\ldots, \gamma_{\alpha}$, and so on, by the rule $\gamma_{\alpha+1}=\varphi(\gamma_{\alpha})$ if $\alpha$ is a successor ordinal, and $\gamma_{\alpha}=\bigcup_{\beta<\alpha}\varphi(\gamma_{\beta})$ for any $\alpha<\tau$.

Put $\gamma^{\ast}=\bigcup_{\alpha<\tau}\gamma_{\alpha}$. Clearly, $|\gamma_{\alpha}|\leq\tau$ for any $\alpha<\tau$ and $\gamma_{\alpha}\subseteq \gamma_{\beta}$ for every $\alpha<\beta$, then $\gamma^{\ast}$ satisfies conditions 1)-4).
\end{proof}

\begin{lemma}\label{l20221}
Let $(G, \tau)$ be an $\tau$-balanced almost topological group, and $U$ an open neighborhood of the neutral element $e$ in $G$. Then there exists a family $\{U_{\alpha}: \alpha<\tau\}$ of open neighborhoods of $e$ such that, for every $\alpha<\tau$, the following conditions are satisfied:

\smallskip
a) $U_{0}\subseteq U$;

\smallskip
b) $U_{\alpha}=U_{\alpha}^{-1}$;

\smallskip
c) there exists $\beta<\tau$ such that $U_{\beta}^{2}\subseteq U_{\alpha}$;

\smallskip
d) for any $x\in G$, there exists $\delta<\tau$ such that $xU_{\delta}x^{-1}\subseteq U_{\alpha}$.
\end{lemma}

\begin{proof}
By Lemma~\ref{l20222}, we only put $\gamma=\{U\}$, and then let $\gamma^{\ast}=\{U_{\alpha}: \alpha<\tau\}$ such that $U_{0}\subseteq U$, as desired.
\end{proof}

\begin{theorem}\label{t20224}
Let $G$ be an $\tau$-balanced almost topological group. Then, for every open neighborhood $U$ of the neutral element $e$ in $G$, there exists a family $\{\rho_{\alpha}: \alpha<\tau\}$ of pre-continuous left-invariant pseudometrics such that the following conditions are satisfied:

\smallskip
(a1) there exists $\alpha<\tau$ such that $\{x\in G: \rho_{\alpha}(e, x)<1\}\subseteq U$;

\smallskip
(a2) $\{x\in G: \rho_{\alpha}(e, x)=0, \alpha<\tau\}$ is a closed invariant subgroup of $G$;

\smallskip
(a3) for any $x$ and $y$ in $G$, $\rho_{\alpha}(e, xy)\leq\rho_{\alpha}(e, x)+\rho_{\alpha}(e, y)$, where $\alpha<\tau$.
\end{theorem}

\begin{proof}
From Lemma~\ref{l20221}, it follows that there exists a family $\{U_{\alpha}: \alpha<\tau\}$ of open neighborhoods of $e$ in $G$ satisfying conditions a)-d) of that lemma. By Lemma~\ref{lll1}, we can find a family $\{N_{\alpha}: \alpha<\tau\}$ of pre-continuous prenorms on $G$ such that the following conditions are satisfied:

\smallskip
a4) for each $\alpha<\tau$, there exists a subsequence $\{U_{\alpha_{i}}: i\in\omega\}$ of $\{U_{\alpha}: \alpha<\tau\}$ such that $U_{\alpha_{0}}=U_{\alpha}$ and $\{x\in G: N_{\alpha}(X)<1/2^{i}\}\subseteq U_{\alpha_{i}}\subseteq\{x\in G: N_{\alpha}(X)<2/2^{i}\};$

For any $x$ and $y$ in $G$, put $\rho_{\alpha}(x, y)=N_{\alpha}(x^{-1}y)$, where $\alpha<\tau$. Since each $N_{\alpha}$ is pre-continuous, it follows that each $\rho_{\alpha}$ is also pre-continuous. Moreover, it is easily checked that each $\rho_{\alpha}$ is left-invariant pseudometric on the set $G$. By a4) and a) of Lemma~\ref{l20221} that a1) holds. Moreover, a3) is obviously satisfied. Now, we only need to prove that a2) holds.

For each $\alpha<\tau$, put $Z_{\alpha}=\{x\in G: N_{\alpha}(x)=0\}$. Then each $Z_{\alpha}$ is closed in $G$, hence $\bigcap_{\alpha<\tau}Z_{\alpha}$ is closed in $G$. Put $Z=\bigcap_{\alpha<\tau}Z_{\alpha}$. Then $Z=\{x\in G: N_{\alpha}(x)=0, \alpha<\tau\}$ is a closed subgroup of $G$. From the definition of each $\rho_{\alpha}$, we have $Z=\{x\in G: \rho_{\alpha}(e, x)=0, \alpha<\tau\}$. We claim that $Z$ is an invariant subgroup of $G$. Indeed, take any $x\in G$. We have to check that $xZx^{-1}=Z$. By a4), it is easily checked that $Z=\bigcap_{\alpha<\tau}U_{\alpha}$, hence it suffices to show that $xZx^{-1}\subseteq U_{\alpha}$ for each $\alpha<\tau$. Fix $\alpha<\tau$. It follows from condition d) of Lemma~\ref{l20221} that there is $\beta<\tau$ such that $xU_{\beta}x^{-1}\subseteq U_{\alpha}$. Since $Z\subseteq U_{\beta}$, we conclude that $xZx^{-1}\subseteq xU_{\beta}x^{-1}\subseteq U_{\alpha}$. Therefore, $Z$ is invariant.
\end{proof}

\begin{theorem}\label{t20225}
Let $G$ be an almost topological group with $inv(G)\leq\tau$. Then, for each open neighborhood $U$ of the neutral element $e$ in $G$, there exists a pre-continuous homomorphism $\pi$ of $G$ onto an almost topological group $H$ with $\chi(H)\leq\tau$ such that $\pi^{-1}(V)\subseteq U$ for some open neighborhood $V$ of the neutral element $e_{H}$ of $H$.
\end{theorem}

\begin{proof}
We continue to use the objects in the proof of Theorem~\ref{t20224}. In particular, we have a family $\{\rho_{\alpha}: \alpha<\tau\}$ of pseudometrics on $G$ constructed above.

Suppose that $H=G/Z$ is the quotient group, and suppose that $\pi$ is the canonical homomorphism of $G$ onto $H$. For any $A, B\in H$ and $\alpha<\tau$, put $d_{\alpha}(A, B)=\rho_{\alpha}(a, b)$, where $a\in A, b\in B$.

\smallskip
{\bf Claim 1:} The definition of each $d_{\alpha}(A, B)$ does not depend on the choice of $a$ in $A$ and $b$ in $B$.

Indeed, fix any $\alpha<\tau$. For any $a\in A$ and $b\in B$, we have $A=aZ$ and $B=bZ$. It suffices to prove that $\rho_{\alpha}(a_{1}, b_{1})=\rho_{\alpha}(a, b)$ for any $a_{1}\in aZ$ and $b_{1}\in bZ$. We may assume that $b=b_{1}$; otherwise we simply repeat the argument twice. Then $a_{1}=az$ for some $z\in Z$. Hence, $N_{\alpha}(z)=N_{\alpha}(z^{-1})=0$, then it follows from \cite[Lemma 3.4.16]{AT} that $$\rho_{\alpha}(a_{1}, b)=N_{\alpha}(z^{-1}a^{-1}b)= N_{\alpha}(z^{-1})+ N_{\alpha}(a^{-1}b)=N_{\alpha}(a^{-1}b)=\rho_{\alpha}(a, b).$$

We also define a family of functions $N_{H}^{\alpha}$ on $H$ by $N_{H}^{\alpha}(A)=N_{\alpha}(a)$ for each $n\in\omega$, $A\in H$ and $a\in A$. It is obvious that each $N_{H}^{\alpha}$ is well-defined.

From the above definitions, for any $\alpha<\tau$ we have that $d_{\alpha}(\pi(a), \pi(b))=\rho_{\alpha}(a, b)$ for any $a, b\in G$, and $N_{H}^{\alpha}(\pi(a)=N_{\alpha}(a)$ for each $a\in G$. Clearly, each $d_{\alpha}$ is a pseudometric; moreover, each $N_{H}^{\alpha}$ is a prenorm on $H$ satisfying the additional conditions as follows:

\smallskip
a5) If $N_{H}^{\alpha}(A)=0$ for any $\alpha<\tau$, then $A$ is the neutral element $e_{H}$ of $H$.

\smallskip
For any $\varepsilon>0$ and $\alpha<\tau$, put $$B_{\alpha}(\varepsilon)=\{x\in G: N_{\alpha}(x)<\varepsilon\},$$  and $$O_{\alpha}(\varepsilon)=\{X\in H: N_{H}^{\alpha}(X)<\varepsilon\}.$$Obviously, we have $\pi(B_{\alpha}(\varepsilon))=O_{\alpha}(\varepsilon)$ for any $\varepsilon>0$ and $\alpha<\tau$. Note that for each $\alpha<\tau$, it is easily checked that the prenorm $N_{\alpha}$ also satisfies the following conditions:

\smallskip
a6) For every $x\in G$ and every $\varepsilon>0$, there exists $\delta>0$ and $\beta<\tau$ such that $xB_{\beta}(\delta)x^{-1}\subseteq B_{\alpha}(\varepsilon)$.

\smallskip
By a6), for each $\varepsilon>0$ and each $X\in H$, there exists $\delta>0$ and $\beta<\tau$ such that $XO_{\beta}(\delta)X^{-1}\subseteq O_{\alpha}(\varepsilon)$.

\smallskip
a7) For any $\varepsilon>0$ and $\alpha<\tau$, we have $O_{\alpha}(\varepsilon)=(O_{\alpha}(\varepsilon))^{-1}$.

\smallskip
a8) For any $\alpha<\tau$ and $\delta>0$, there exists $\beta<\tau$ and $\varepsilon>0$ such that $(O_{\beta}(\varepsilon))^{2}\subseteq O_{\alpha}(\delta).$

\smallskip
a9) $\{e_{H}\}=\bigcap_{\alpha<\tau, m\in\omega}O_{\alpha}(1/2^{m})$.

Let $\mathscr{F}_{H}$ be the pre-topology generated by the family $\{d_{\alpha}: \alpha<\tau\}$ of pseduometrics on $H$. We will prove that $H$ with this pre-topology is an almost topological group.

Indeed, since each pseduometric $d_{\alpha}$ is left-invariant, it suffices to prove that the family $\{O_{\alpha}(1/2^{m}): \alpha<\tau, m\in\omega\}$ satisfies the axioms in Theorem~\ref{t0} for a pre-base of a group pre-topology at the neutral element. And this is exactly what conditions a6)-a9) guarantee, as is routinely checked. Moreover, by a7), $H$ with the pre-topology $\mathscr{F}_{H}$ is an almost topological group. Clearly, $\chi(H)\leq\tau$.

Finally, since $\pi(B_{\alpha}(\varepsilon))=O_{\alpha}(\varepsilon))$ for any $\alpha<\tau$ and $\varepsilon>0$, it follows that $\pi$ is a pre-continuous at the neutral element, hence $\pi$ is pre-continuous. Moreover, if $x\in G$, $X=\pi(x)$, $\varepsilon>0$ and $\alpha<\tau$, then $N_{\alpha}(x)<\varepsilon$ if and only if $N_{H}^{\alpha}(X)<\varepsilon$. Hence, $\pi^{-1}(O_{\alpha}(\varepsilon))=B_{\alpha}(\varepsilon)$ for any $\varepsilon>0$. In particular, $\pi^{-1}(O_{0}(1))=B_{0}(1)\subseteq U_{0}\subseteq U.$
\end{proof}

\begin{corollary}\label{c20221}
Let $G$ be an $\omega$-narrow almost topological group. Then, for each open neighborhood $U$ of the neutral element $e$ in $G$, there exists a pre-continuous homomorphism $\pi$ of $G$ onto a second-countable almost topological group $H$ such that $\pi^{-1}(V)\subseteq U$ for some open neighborhood $V$ of the neutral element $e_{H}$ of $H$.
\end{corollary}

\begin{proof}
By Theorem~\ref{t20225}, there exists a pre-continuous homomorphism $\pi$ of $G$ onto a first-countable almost topological group $H$ and an open neighborhood $V$ of the neutral element in $H$ such that $\pi^{-1}(V)\subseteq U.$ Then it follows from Proposition~\ref{p20223} that $H$ is $\omega$-narrow, hence $H$ is second-countable by Proposition~\ref{p20224}.
\end{proof}

\begin{corollary}\label{t20225}
Let $G$ be an almost topological group with $inv(G)\leq\omega$. Then, for each open neighborhood $U$ of the neutral element $e$ in $G$, there exists a pre-continuous homomorphism $\pi$ of $G$ onto a first-countable almost topological group $H$ such that $\pi^{-1}(V)\subseteq U$ for some open neighborhood $V$ of the neutral element $e_{H}$ of $H$.
\end{corollary}

\begin{corollary}
Let $G$ be a $\omega$-narrow almost topological group. Then, for each open neighborhood $U$ of the neutral element $e$ in $G$, there exists a pre-continuous homomorphism $\pi$ of $G$ onto a second-countable almost topological group $H$ such that $\pi^{-1}(V)\subseteq U$ for some open neighborhood $V$ of the neutral element $e_{H}$ of $H$.
\end{corollary}

\begin{theorem}\label{t20228}
Each almost topological group $G$ with $inv(G)\leq\tau$ can be embedded as a subgroup into a pre-topological product of almost topological groups of character $\leq\tau.$
\end{theorem}

\begin{proof}
Let $\mathscr{B}=\{U_{i}: i\in I\}$ of all open neighborhoods of the neutral element $e$ in $G$. By Theorem~\ref{t20225}, for each $i\in I$ there exists a pre-continuous homomorphism $\pi_{i}$ of $G$ onto an almost topological group $H_{i}$ with $\chi(H_{i})\leq\tau$ such that $\pi_{i}^{-1}(V_{i})\subseteq U_{i}$ for some open neighborhood $V_{i}$ of the neutral element in $H_{i}$. Suppose that $\prod=\prod_{i\in I}H_{i}$ is the pre-topological product of the pre-topological groups $H_{i}$'s, and suppose that $\varphi: G\rightarrow\prod$ is the diagonal product of the homomorphism $\pi_{i}$, where $i\in I$. Obviously, $\varphi$ is the pre-continuous homomorphism of $G$ to $\prod$. Now it suffices to prove that $\varphi$ is a pre-topological embedding.

Let $H=\varphi(G)$. Clearly, it is easily checked that $H$ is an almost topological group. By Theorem~\ref{t20227}, $\varphi: G\rightarrow H$ is a bijective mapping. Take an arbitrary open neighborhood $U$ of $e$ in $G$. Hence there exists $i\in I$ such that $U_{i}\subseteq U$ and then $\pi^{-1}_{i}(V_{i})\subseteq U_{i}$ by the choice of the open neighborhood $V_{i}$ of the neutral element in $H_{i}$. Let $p_{i}$ be the projection of $\prod$ onto the factor $H_{i}$. Then the set $W=p_{i}^{-1}(V_{i})$ is an open neighborhood of the neutral element in $\prod$. Obviously, we have $\pi_{i}=p_{i}\circ\varphi$ for each $i\in I$, hence $\varphi^{-1}(W)=\pi_{i}^{-1}(V_{i})\subseteq U_{i}\subseteq U.$ Therefore, $V=W\cap H$ is an open neighborhood of the neutral element in $H$ satisfying $\varphi^{-1}(O)\subseteq U.$

Therefore, $\varphi: G\rightarrow H$ is a pre-topological isomorphism.
\end{proof}

Now we can prove one of the main results in this section. Indeed, the following theorem generalizes the well-known Guran's Theorem.

\begin{theorem}\label{t202212}
An almost topological group $G$ is $\tau$-narrow if and only if $G$ can be embedded as a subgroup of a pre-topological product of almost topological groups of weight less than or equal to $\tau$.
\end{theorem}

\begin{proof}
By Proposition~\ref{u0}, each almost topological group of weight $\leq\tau$ is $\tau$-narrow. Then, by Theorem~\ref{t202210} and Proposition~\ref{p20225}, each subgroup of a pre-topological product $\prod_{i\in I}H_{i}$ of almost topological groups is $\tau$-narrow provided that $w(H_{i})\leq\tau$ for each $i\in I$.

Conversely, suppose that almost topological group $G$ is $\tau$-narrow. By Theorem~\ref{t20228}, Propositions~\ref{p20224} and~\ref{p20223}, it is easily seen that we can identify $G$ with a subgroup of a pre-topological product $\prod=\prod_{i\in I}H_{i}$ of almost topological groups $H_{i}$ satisfying $w(H_{i})\leq\tau$ for each $i\in I$.
\end{proof}

The following theorem gives another characterization of $\tau$-narrow almost topological groups. First, we give an obviously technical lemma.

\begin{lemma}\label{l202211}
For each $\alpha<\tau$, let $H_{\alpha}$ be a pre-topological space with $w(H_{\alpha})\leq\tau$. Then the pre-topological product $H=\prod_{\alpha<\tau}H_{\alpha}$ satisfies $c(H)\leq\tau$.
\end{lemma}

\begin{theorem}
An almost topological group $G$ is $\tau$-narrow if and only if it can be embedded as a subgroup into an almost topological group $H$ satisfying $c(H)\leq\tau$.
\end{theorem}

\begin{proof}
By Theorems~ \ref{t202211} and~\ref{t202210}, each subgroup of an almost topological group $H$ with $c(H)\leq\tau$ is $\tau$-narrow. Conversely, by Theorem~\ref{t202212}, a $\tau$-narrow almost topological group $G$ is pre-topologically isomorphic to a subgroup of a pre-topological product $H$ of almost topological groups $H_{i}$ with $w(H_{i})\leq\tau$. By Lemma~\ref{l202211}, $c(H)\leq\tau$, as desired.
\end{proof}

A pre-topological space $X$ is said to be {\it $\sigma$-compact} if $X=\bigcup_{n\in\omega}X_{n}$, where each $X_{n}$ is compact. Moreover, a pre-topological group is said to be {\it $k$-separable} if it has a dense $\sigma$-compact subgroup.

\begin{lemma}\label{l2022111}
The $\sigma$-product of any family of compact $T_{2}$ pre-topological spaces is $\sigma$-compact.
\end{lemma}

\begin{proof}
Let $\{X_{\alpha}: \alpha\in I\}$ be a family of compact $T_{2}$ pre-topological spaces and let $X=\prod_{\alpha\in I}X_{\alpha}$. Suppose that $Y$ be the corresponding $\sigma$-product with center at $b\in X$. Then, for each $y\in Y$, only finitely many coordinates $y_{\alpha}$ of $y$ are distinct from the corresponding coordinates $b_{\alpha}$ of $b$. Denote by $r(y)$ the number of coordinates of a point $y\in Y$ distinct from those of $b$. For each $n\in\omega$, put $$Y_{n}=\{y\in Y: r(y)\leq n\}.$$ Obviously, we have $Y=\bigcup_{n\in\omega}Y_{n}$, and each $Y_{n}$ is closed in the product space $X$ since each $X_{\alpha}$ is $T_{2}$. Therefore, $Y$ is $\sigma$-compact.
\end{proof}

By Lemma~\ref{l2022111}, we have the following lemma by a similar proof of \cite[Proposition 1.6.41]{AT}, so we left the proof for the reader.

\begin{lemma}\label{l202210}
The $\sigma$-product of any family of $\sigma$-compact $T_{2}$ pre-topological spaces is $\sigma$-compact.
\end{lemma}

The following lemma is obvious.

\begin{lemma}\label{l202222}
The $\sigma$-product of any family $\mathscr{A}$ of pre-topological spaces is dense in the product of the family $\mathscr{A}$.
\end{lemma}

The following theorem gives a generalization of well-known theorem of Pestov's.

\begin{theorem}
The class of $\omega$-narrow almost topological groups coincides with the class of subgroups of $k$-separable almost topological groups.
\end{theorem}

\begin{proof}
Let $G$ be a $k$-separable almost topological group. Then $G$ has a dense $\sigma$-compact subgroup. Since each $\sigma$-compact almost topological group is Lindel\"{o}f, it follows from Corollary~\ref{c20222} that $H$ is $\omega$-narrow. From Theorem~\ref{t202213}, it follows that $G$ is $\omega$-narrow, hence each subgroup of $G$ is also $\omega$-narrow by Theorem~\ref{t202210}.

Conversely, suppose that $G$ is an arbitrary $\omega$-narrow almost topological group. From Theorem~\ref{t202212}, $G$ can be embedded as a subgroup of a pre-topological product $\prod=\prod_{i\in I}H_{i}$ of second-countable almost topological groups $H_{i}$'s. For each $i\in I$, we can fix a countable dense subgroup $D_{i}$ of $H_{i}$, and then let $D$ be a $\sigma$-product of $\prod_{i\in I}D_{i}$. By Lemmas~\ref{l202210} and~\ref{l202222}, $D$ is $\sigma$-compact and dense in $\prod_{i\in I}D_{i}$.
\end{proof}

The following proposition shows that, in the class of Abelian almost topological groups $G$, pre-continuous homomorphic images $H$ with $\chi(H)\leq\omega$ of a given almost topological group $G$ determine whether $G$ is $\tau$-narrow or not.

\begin{proposition}
Suppose that $G$ is an Abelian almost topological group and suppose that each pre-continuous homomorphic image $H$ of $G$ with $\chi(H)\leq\omega$ is $\tau$-narrow. Then the almost topological group $G$ is also $\tau$-narrow.
\end{proposition}

\begin{proof}
Take an arbitrary open neighborhood $U$ of the neutral element $e$ in $G$. Since $G$ is an almost topological group, there exists a sequence $\{U_{n}: n\in\omega\}$ of open symmetric neighborhoods of $e$ in $G$ such that $U_{0}\subseteq U$ and $U_{n+1}^{2}\subseteq U_{n}$ for every $n\in\omega$. Put $N=\bigcap_{n\in\omega}U_{n}$ is a closed subgroup of $G$. Since $G$ is Abelian, the set of all
cosets $G/N$ is a group. Let $\pi: G\rightarrow G/N$ be the natural homomorphism. It is easily checked that the family $\{\pi(U_{n}): n\in\omega\}$ is a pre-base for a Hausdorff almost topological group pre-topology $\mathscr{F}$ on $G/N$ at the neutral element of this group. Let $H=(G/N, \mathscr{F})$. Then $\pi: G\rightarrow H$ is pre-continuous and $\chi(H)\leq\omega$, hence $H$ is $\tau$-narrow. Let $V=\pi(U_{1})$. Then there exists a subset $K\subseteq H$ with $|K|\leq\tau$ such that $KV=H$. Let $F$ be any subset of $G$ such that $\pi(F)=K$ and $|F|\leq\tau$. We conclude that $FU=G$. Indeed, take any point $x\in G$. Hence $\pi(x)\in bV$ for some $b\in K$. Hence we can choose an element $a\in F$ such that $\pi(a)=b$, then $\pi(x)\in bV=\pi(aU_{1})$, hence it follows that $$x\in\pi^{-1}\pi(aU_{1})=aU_{1}N\subseteq aU_{1}U_{1}\subseteq aU_{0}\subseteq aU\subseteq FU.$$ Therefore, $G$ is $\tau$-narrow.
\end{proof}

\section{The precompactness in pre-topological groups}

In this section, some basic properties about precompactness in pre-topological groups are investigated.

\begin{definition}
A semi-pre-topological group is called {\it left precompact} (resp. {\it right precompact}) if, for each open neighborhood $U$ of the neutral element in $G$, there exists a finite subset $F$ of $G$ such that $G= FU$ (resp. $G=UF$)). If $G$ is left precompact and right precompact then $G$ is called precompact.
\end{definition}

Compare with Proposition~\ref{p202210}, we also have the following proposition.

\begin{proposition}\label{p2022101}
The following conditions are equivalent for a quasi-pre-topological group $G$.

\smallskip
(1) $G$ is precompact;

\smallskip
(2) For every open neighbourhood $V$ of $e$ in $G$, there exists a finite subset $B\subseteq G$ such that $G=VB$;

\smallskip
(3) For every open neighbourhood $V$ of $e$ in $G$, there exists a finite subset $C\subseteq G$ such that $CV=VC=G$.
\end{proposition}

The proof of the following proposition is obvious, thus we omit it.

\begin{proposition}\label{p20222022}
If $f$ is a pre-continuous homomorphism of a precompact pretopological group $G$ onto a pre-topological group $H$, then $H$ is also precompact.
\end{proposition}

It is obvious that a discrete pre-topological group is precompact if and only if it is finite. Clearly, each precompact pre-topological group is $\omega$-narrow and each compact pre-topological group is precompact. Moreover, we have the following more general fact.

\begin{proposition}
Each feebly compact almost topological group is precompact.
\end{proposition}

\begin{proof}
Take an arbitrary open symmetric neighborhood $V$ of the neutral element $e$ in $G$. Then there exists an open symmetric neighborhood $U$ of $e$ in $G$ such that $U^{4}\subseteq V$. From the proof of Theorem~\ref{t202211}, we can find a maximal $V$-disjoint set $A$ of $G$. From Lemma~\ref{r1}, it follows that the family of open sets $\{xU: x\in A\}$ is discrete in $G$, hence $A$ is finite since $G$ is feebly compact. Then, by the maximality of $A$, we have $G=AV$. Hence $G$ is precompact.
\end{proof}

In order to discuss the relation of the precompactness between pre-topological groups and their subgroups, we introduce the following concept and prove some lemmas.

A subset $B$ of a semi-pre-topological group $G$ is called {\it precompact} in $G$ if, for each neighborhood $U$ of the neutral element in $G$, there exists a finite set $F\subseteq G$ such that $B\subseteq FU$ and $B\subseteq UF$.

\begin{lemma}\label{l2022100}
Let $B$ be a precompact subset of a pre-topological group $G$ and $D$ is dense in $B$. Then, for each open neighborhood $U$ of the neutral element in $G$, there exists a finite set $K\subseteq D$ such that $B\subseteq KU$ and $B\subseteq UK$.
\end{lemma}

\begin{proof}
Take an arbitrary open neighborhood $U$ of the neutral element in $G$. Then there exists open neighborhoods $V_{1}$ and $V_{2}$ such that $V_{1}V_{2}\subseteq U$. Since $B$ is precompact, we can find a finite subset $F$ of $G$ such that $B\subseteq FV_{1}^{-1}$ and $B\subseteq V_{1}^{-1}F$. For each $x\in F$, if $B\cap xV_{1}^{-1}\neq\emptyset$, then we can pick a point $y_{x}\in D\cap xV_{1}^{-1}$. Put $$K_{1}=\{y_{x}: x\in F, B\cap xV_{1}^{-1}\neq\emptyset\}.$$Then $K_{1}$ is finite and is contained in $D$. We claim that $B\subseteq K_{1}U$. Indeed, take any $b\in B$. Then there is $x\in F$ such that $b\in xV_{1}^{-1}$. Hence $b\in B\cap xV_{1}^{-1}\neq\emptyset$, thus $y_{x}\in xV_{1}^{-1}$, then $y_{x}^{-1}x\in V_{1}$. Therefore, it follows that $$b\in xV_{2}=y_{x}(y_{x}^{-1}x)V_{2}\subseteq y_{x}V_{1}V_{2}\subseteq y_{x}U\subseteq K_{1}U.$$ Thus $B\subseteq K_{1}U.$ Similarly, we can find a finite subset $K_{2}$ of $S$ such that $B\subseteq UK_{2}$. Now put $K=K_{1}\cup K_{2}$. Then the finite set $K$ is as required.
\end{proof}

\begin{proposition}\label{p20221011}
Each subgroup $H$ of a precompact pre-topological group $G$ is a precompact pre-topological group.
\end{proposition}

\begin{proof}
Let $U$ be an arbitrary open neighborhood $U$ of the neutral element in $H$. Then there exists an open neighborhood $V$ of $e$ in $G$ such that $V\cap H=U$. Since $G$ is precompact, it follows that $H$ is a precompact subset of $G$.  By Lemma~\ref{l2022100}, there is a finite subset $F$ of $H$ such that $H\subseteq FV$ and $H\subseteq VF$. We conclude that $H\subseteq FU$ and $H\subseteq UF$. Indeed, for each $h\in H$, there exist $x\in F$ and $y\in V$ such that $h=xy$. Since $H$ is a subgroup, it follows that $y=x^{-1}h\in V\cap H=U$, which implies that $h\in FU$. Thus $F\subseteq FU$. Similarly, we also have $H\subseteq UF$. Therefore, $H$ is a precompact pre-topological group.
\end{proof}

\begin{lemma}
Suppose that $B$ is a subset of pre-topological group $G$ such that $B$ contains a dense precompact subset. Then $B$ is also precompact in $G$. Therefore, the closure of a precompact subset of $G$ is precompact in $G$.
\end{lemma}

\begin{proof}
Let $D$ be a dense precompact subset of $B$. Take an arbitrary open neighborhood $U$ of the neutral element $e$ in $G$. Then there exist open neighborhoods $V_{1}$ and $V_{2}$ of $e$ in $G$ such that $V_{1}V_{2}\subseteq U$. Because $D$ is precompact in $G$,  there is a finite subset $F$ such that $D\subseteq FV_{1}$ and $D\subseteq V_{1}F$. We conclude that $B\subseteq FU\cap UF$. Indeed, take any $b\in B$. Since $D$ is dense in $B$, then $bV_{2}^{-1}\cap D\neq\emptyset$, hence we can pick a point $y_{b}\in bV_{2}^{-1}\cap D$. Hence $y_{b}\in xV_{1}$ for some $x\in F$ since $D\subseteq FV_{1}$, then $b\in y_{b}V_{2}\subseteq xV_{1}V_{2}\subseteq xU$. Therefore, we have $B\subseteq FU$. Similarly, we have $B\subseteq UF$.
\end{proof}

\begin{corollary}
If a pre-topological group $G$ contains a dense precompact subgroup, then $G$ is also precompact.
\end{corollary}

The following proposition, though quite easy, is nevertheless, rather interesting.

\begin{proposition}
The dispersion character of a precompact pre-topological group $G$ is equal to its cardinality.
\end{proposition}

\begin{proof}
Let $U$ be an arbitrary open neighborhood of the neutral element $e$ in $G$. Since $G$ is precompact, there exists a finite subset $F$ such that $FU=G$, hence $|U|=|G|$.
\end{proof}

We say that a pre-topological group is {\it finite-balanced} if for any open neighbourhood $U$ of the neutral element $e$ of $G$, there exists a finite family $\gamma$ of open neighbourhoods of $e$ such that for each $x\in G$ there exists $V\in \gamma$ satisfying $xVx^{-1}\subseteq U$. Any such family $\gamma$ will be called {\it subordinated} to $U$.

By a similar proof of Proposition~\ref{ppp1}, we have the following proposition.

\begin{theorem}
Each precompact almost topological group $G$ is finite-balanced.
\end{theorem}

\begin{theorem}~\label{t20222022}
If $B_{i}$ is a precompact subset of a pre-topological group $G_{i}$ for every $i\in I$, then the set $B=\prod_{i\in I}B_{i}$ is precompact in the pre-topological product $G=\prod_{i\in I}G_{i}$.
\end{theorem}

\begin{proof}
Let $U$ be an arbitrary open neighborhood of the neutral element in $G=\prod_{i\in I}G_{i}$. Then there exist a finite subset $C\subseteq I$ and open neighborhood $U_{i}$ of the neutral element $e_{i}$ of $G_{i}$ for each $i\in C$ such that $\prod_{i\in C}U_{i}\times \prod_{i\in I\setminus C}G_{i}\subseteq U$. For each $i\in C$, there exists a finite subset $F_{i}$ such that $B_{i}\subseteq F_{i}U_{i}$ and $B_{i}\subseteq U_{i}F_{i}$. Put $F=\prod_{i\in C}F_{i}\times \prod_{i\in I\setminus C}\{e_{i}\}$. Then $\prod_{i\in I}B_{i}\subseteq F(\prod_{i\in C}U_{i}\times \prod_{i\in I\setminus C}G_{i})\subseteq FU$ and $\prod_{i\in I}B_{i}\subseteq (\prod_{i\in C}U_{i}\times \prod_{i\in I\setminus C}G_{i})F\subseteq UF$. The proof is completed.
\end{proof}

\begin{corollary}
The product of a family of precompact pre-topological groups is a precompact pre-topological group.
\end{corollary}

By Proposition~\ref{p20222022} and Theorem~\ref{t20222022}, we have the following corollary.

\begin{corollary}
Let $A$ and $B$ be precompact subsets of a pre-topological group $G$. Then the sets $A^{-1}$, $B^{-1}$ and $AB$ are precompact in $G$.
\end{corollary}

Next, we discuss the pre-Ra\v{i}kov completion of a pre-topological group. Let $(G, \tau)$ be a pre-topological group, and let $(G, \tau^{\ast})$ be the co-reflexion group topology of $(G, \tau)$. Then $(G, \tau^{\ast})$ has a Ra\v{i}kov completion $(\rho G, \rho\tau^{\ast})$. We say that the  pre-topology $\sigma$ on $\rho G$ is {\it the pre-Ra\v{i}kov completion} of $(G, \tau)$ if $\sigma$ has a pre-base $\{gU: e\in U\cap G\in\tau, g\in\rho G, U\in\rho\tau^{\ast}\}$, where $e$ is the neutral element of $\rho G$.
For convenience, we denote the pre-Ra\v{i}kov completion of $(G, \tau)$ by pre-$\rho G$.

\begin{theorem}
If $G$ is a pre-topological group, then the pre-Ra\v{i}kov completion pre-$\rho G$ of $G$ is a homogeneous pre-topological space.
\end{theorem}

\begin{proof}
Let $G$ be a pre-topological group. Take any $g, h\in\rho G$. Then the left translation mapping $l: \rho G\rightarrow \rho G$, defined by $l(x)=hg^{-1}(x)$ for each $x\in\rho G$, is a pre-homeomorphic mapping. Hence pre-$\rho G$ of $G$ is a homogeneous pre-topological space.
\end{proof}

\begin{theorem}\label{t20221000}
Let $(G, \tau)$  be a pre-topological group. Then the pre-Ra\v{i}kov completion pre-$\rho G$ of $(G, \tau)$ satisfies the following conditions:

\smallskip
(i) $\sigma$ is a subbase for $\rho\tau^{\ast}$;

\smallskip
(ii) $(G, \tau)$ is a dense pre-topological subgroup of pre-$\rho G$.

\smallskip
(iii) For each open neighborhood $U$ of the neutral element in pre-$\rho G$, we have $U^{-1}\in\sigma$.

\smallskip
(iv) For each open neighborhood $U$ of the neutral element in pre-$\rho G$, there exist open neighborhoods $V_{1}$ and $V_{2}$ of the neutral element $e$ in pre-$\rho G$ such that $V_{1}V_{2}\subseteq U$.

\smallskip
(v) If $G$ is an almost topological group, then, for each open neighborhood $U$ of the neutral element in pre-$\rho G$ and $g\in \rho G$, there exists an open neighborhood $O$ of the neutral element in pre-$\rho G$ such that $gOg^{-1}\subseteq U$.
\end{theorem}

\begin{proof}
(i) It suffices to prove that the family $\mathscr{B}=\{U\in\rho\tau^{\ast}: e\in U\cap G\in\tau\}$ is a subbase at the neutral element in $(\rho G, \rho\tau^{\ast})$. Indeed, take any open neighborhood $U$ of the neutral element of $\rho G$ in $(\rho G, \rho\tau^{\ast})$. Since $U\cap G$ is an open neighborhood of $e$ in  $(G, \tau^{\ast})$, it follows from Proposition~\ref{p20221} that there exist finitely many open neighborhoods $U_{1}, \ldots, U_{n}$ of the neutral element of $G$ in $(G, \tau)$ such that $\bigcap_{i=1}^{n}U_{i}\subseteq U\cap G$. For each $i\leq n$, because $U_{i}$ is open in $(G, \tau^{\ast})$, there exists an open neighborhood $W_{i}\subseteq U$ of the neutral element in $(\rho G, \rho\tau^{\ast})$ such that $W_{i}\cap G=U_{i}$. Therefore, $g\in g\bigcap_{i=1}^{n}W_{i}\subseteq gU$ and $W_{i}\in\mathscr{B}$ for each $i\leq n$.

\smallskip
(ii) Clearly, $G$ is dense in pre-$\rho G$ since $(G, \tau^{\ast})$ is dense in $(\rho G, \rho\tau^{\ast})$. From our definition, we have $\sigma|_{G}=\tau$.
Therefore, $(G, \tau)$ is a dense pre-topological subgroup of pre-$\rho G$.

\smallskip
(iii) It is obvious.

\smallskip
(iv) Take any open neighborhood $U$ of the neutral element in $\sigma$. Then $U\cap G\in\tau$, hence there exist open neighborhoods $W_{1}$ and $W_{2}$ of the neutral element $e$ in $(G, \tau)$ such that $W_{1}W_{2}\subseteq U$. Then $\overline{W_{1}W_{2}}\subseteq \overline{U}$ in $(\rho G, \rho\tau^{\ast})$, hence $\overline{W_{1}}\overline{W_{2}}\subseteq \overline{U}$, thus $$\mbox{int}(\overline{W_{1}})\mbox{int}(\overline{W_{2}})\subseteq \mbox{int}(\overline{W_{1}}\overline{W_{2}})\subseteq \mbox{int}(\overline{U})$$ in $(\rho G, \rho\tau^{\ast})$. Put $\mbox{int}(\overline{W_{1}})=V_{1}, \mbox{int}(\overline{W_{2}})=V_{2}$. Clearly,  $V_{1}\cap G=W_{1}, V_{2}\cap G=W_{2}$ and $\mbox{int}(\overline{U})=U$. Hence both $V_{1}$ and $V_{2}$ are open neighborhoods of the neutral element in pre-$\rho G$.

\smallskip
(v) Take any open neighborhood $U$ of the neutral element in pre-$\rho G$ and $g\in \rho G$. From the proof of (iv) above, there exist open symmetric neighborhoods $W$ of the neutral element in pre-$\rho G$ such that $W^{3}\subseteq U$. By (ii), $G$ is dense in pre-$\rho G$, hence $Wg\cap G\neq\emptyset$ since $(G, \tau^{\ast})$ is dense in $(\rho G, \rho\tau^{\ast})$ and $W$ is open in $(\rho G, \rho\tau^{\ast})$, then there exists $h\in G$ such that $g\in W^{-1}h=Wh$. Then there exists an open neighborhood $V$ of the neutral element in $(G, \tau)$ such that $hVh^{-1}\subseteq W\cap G$, hence $h\overline{V}h^{-1}\subseteq \overline{W\cap G}$ in $(\rho G, \rho\tau^{\ast})$, then $h\mbox{int}(\overline{V})h^{-1}\subseteq \mbox{int}(\overline{W\cap G})$ in $(\rho G, \rho\tau^{\ast})$. Put $O=\mbox{int}(\overline{V})$. Since $O\cap G=V$ and $\mbox{int}(\overline{W\cap G})=W$ in $(\rho G, \rho\tau^{\ast})$, it follows that $O\in\sigma$ and $$gOg^{-1}\subseteq WhOh^{-1}W\subseteq WWW\subseteq U.$$
\end{proof}

By Theorems~\ref{t0} and~\ref{t20221000}, it is easily concluded that the following corollary holds.

\begin{corollary}~\label{cc2022}
Let $G$ be an almost topological group.  If the following condition ($\star$) holds, then pre-$\rho G$ is an almost topological group.

\smallskip
($\star$) For each open neighborhood $U$ of the neutral element in pre-$\rho G$, if $g\in U$, then there exists an open neighborhood $V$ of the neutral element in pre-$\rho G$ such that $gV\subseteq U$
\end{corollary}

\begin{question}\label{q20222}
Let $G$ be an almost topological group. Does pre-$\rho G$ satisfy the condition ($\star$) in Corollary~\ref{cc2022}.
\end{question}

The following theorem gives a partial answer to Question~\ref{q20222}.

\begin{theorem}\label{t202299}
Let $(G, \tau)$ be a pre-topological group. If $(G, \tau^{\ast})$ is locally compact, then pre-$\rho G=G$.
\end{theorem}

\begin{proof}
Since $(G, \tau^{\ast})$ is locally compact, it follows from \cite[Theorem 3..6.24]{AT} that $\rho G$ is just $(G, \tau^{\ast})$. Therefore, pre-$\rho G=G$.
\end{proof}

\begin{corollary}
Each finite pre-topological group is pre-Ra\v{i}kov complete.
\end{corollary}

It is well known that the closure of each precompact subset $B$ in a topological group $G$ is compact in the Ra\v{i}kov completion $\rho G$ of $G$. However, in the class of pre-topological groups, we have the following example.

\begin{example}
There exists a precompact, non-compact pre-topological group $G$ such that pre-$\rho G=G$.
\end{example}

\smallskip
Indeed, let $G$ be the unit circle, that is, $G=\{|z|=1: Z\in \mathbb{C}\}$. For each $0<\theta<\pi$, let $U_{\theta}=e^{i\theta}$ and $U_{-\theta}=e^{-i\theta}$. Let $\mathscr{A}=\{U_{\theta}: 0<\theta<\pi\}\cup \{U_{-\theta}: 0<\theta<\pi\}$. Then $\mathscr{A}$ satisfies the conditions (1)-(4) of Theorem~\ref{t0}, hence it follows from Theorem~\ref{t2022} that the pre-topology $\tau$, generated by the family $\mathscr{A}$, is a pre-topological group topology on $G$. Clearly, $G$ is precompact and non-compact. However, $(G, \tau^{\ast})$ is discrete, thus it is locally compact, then pre-$\rho G=G$ by Theorem~\ref{t202299}.

\begin{remark}
The pre-topological group $(G, \tau)$ in (2) of Example~\ref{d0} is a second-countable and non-precompact pre-topological group. However, $(G, \tau^{\ast})$ is locally compact, hence it follows from \cite[Theorem 3..6.24]{AT} that $\rho G$ is just $(G, \tau^{\ast})$. Therefore, pre-$\rho G=G$ by Theorem~\ref{t202299}.
\end{remark}

\begin{theorem}
Let $G$ be a pre-topological group and $A$ be a subset of $G$. Then $A$ is a precompact subset of $G$ if and only if the closure of $A$ in the pre-Ra\v{i}kov completion pre-$\rho G$ is precompact.
\end{theorem}

\begin{proof}
By the proof of Proposition~\ref{p20221011}, the sufficiency is obvious. It suffices to prove the necessity. Let $A$ be a precompact subset of $G$. Take any open neighborhood $U$ of the neutral element in pre-$\rho G$. By (iV) of Theorem~\ref{t20221000}, there exist open neighborhoods $V_{1}$ and $V_{2}$ in pre-$\rho G$ such that $V_{1}V_{2}\subseteq U$, then $\overline{V_{1}}\subseteq U$ by the proof of Theorem~\ref{t2022111}. Since $U\cap G$ is open in $G$, there exist finite subset $F\subseteq G$ such that $A\subseteq F(V_{1}\cap G)$ and $A\subseteq (V_{1}\cap G)F$. Then $\overline{A}\subseteq\overline{F(V_{1}\cap G)}=F\overline{V_{1}}$ and $\overline{A}\subseteq\overline{(V_{1}\cap G)F}=F\overline{V_{1}}$. Since $F\overline{V_{1}}\subseteq FU$ and $\overline{V_{1}}F\subseteq UF$, it follows that the closure of $A$ in pre-$\rho G$ is precompact.
\end{proof}

\begin{corollary}
A pre-topological group $G$ is precompact if and only if the pre-Ra\v{i}kov completion pre-$\rho G$ is precompact.
\end{corollary}

\begin{question}
Let $G$ be a pre-topological group. If the pre-Ra\v{i}kov completion pre-$\rho G$ of $G$ is compact, is pre-$\rho G=G$?
\end{question}

\begin{theorem}
Each pre-topological product $G=\prod_{i\in I}G_{i}$ of pre-Ra\v{i}kov complete pre-topological spaces is pre-Ra\v{i}kov complete.
\end{theorem}

\begin{proof}
For each $i\in I$, let $\tau_{i}$ be the pre-topology of $G_{i}$; since $G_{i}$ is pre-Ra\v{i}kov complete, it follows that $G_{i}$=pre-$\rho G_{i}$, hence $(G, \tau^{\ast})$ is Ra\v{i}kov complete. Therefore, it follows from \cite[Theorem 3.6.22]{AT} that the topological product $G=\prod_{i\in I}(G_{i}, \tau_{i}^{\ast})$ is Ra\v{i}kov complete, then $G=\prod_{i\in I}G_{i}$ is pre-Ra\v{i}kov complete.
\end{proof}

\begin{corollary}
Let $G=\prod_{i\in I}G_{i}$ be a product of pre-topological groups. Then pre-$\rho G$ is pre-topologically isomorphic to the pre-product pre-topological space $\prod_{i\in I}$pre-$\rho G_{i}$.
\end{corollary}

Finally, we discuss some applications of the precompactness in pre-topological groups. First, we say that
a pre-topological space $X$ is {\it resolvable} if there exist dense disjoint subsets $A$ and $B$.

\begin{proposition}\label{p202288}
If a subgroup $H$ of a pre-topological group $G$ is resolvable, then so is $G$.
\end{proposition}

\begin{proof}
Let $A$ and $B$ be two dense disjoint subsets in $H$, and let $\mathscr{C}=\{x_{\alpha}H: x_{\alpha}\in G, \alpha\in I\}$ be all the cosets of $H$ in $G$ such that $x_{\alpha}H\cap x_{\beta}=\emptyset$ for any $\alpha\neq\beta$. Put $D_{1}=\bigcup\{x_{\alpha}A: \alpha\in I\}$ and $D_{2}=\bigcup\{x_{\alpha}B: \alpha\in I\}$.  Then it is easily checked that $D_{1}$ and $D_{2}$ are disjoint dense subsets of $G$.
\end{proof}

The following proposition is obvious, so we omit the proof.

\begin{proposition}\label{p202277}
If a pre-topological group $G$ contains a proper dense subgroup, then $G$ is resolvable.
\end{proposition}

\begin{proposition}
If a pre-topological group $G$ contains a non-closed subgroup, then $G$ is resolvable.
\end{proposition}

\begin{proof}
Let $H$ be a non-closed subgroup of $G$. Then $\overline{H}$ is a pre-topological subgroup by Proposition~\ref{p202266}. If $\overline{H}=G$, then $G$ is resolvable by Proposition~\ref{p202277} since $H$ is non-closed. Now we assume that $\overline{H}\neq G$. Since $H\neq\overline{H}$, it follows that $\overline{H}$ is resolvable by Proposition~\ref{p202277}. Then, from Proposition~\ref{p202288}, it follows that $G$ is resolvable.
\end{proof}

\begin{lemma}\label{l2022110}\cite{MP1996}
For any infinite group $G$, there exists a disjoint family $\mathscr{A}$ of cardinality $|G|$ of subsets of $G$ such that, for each $A\in\mathscr{A}$ and any finite subset $K$ of $G$, $AK\neq G$ and $(G\setminus A)K\neq G$.
\end{lemma}

\begin{theorem}\label{t202224}
For any infinite group $G$, there exists a disjoint family $\mathscr{A}$ of cardinality $|G|$
of subsets of $G$ which are dense in any precompact pre-topological group on $G$.
\end{theorem}

\begin{proof}
Let $G$ be a precompact pre-topological group on $G$. By Lemma~\ref{l2022110}, there exists a disjoint family $\mathscr{A}$ of cardinality $|G|$ of subsets of $G$ such that, for each $A\in\mathscr{A}$ and any finite subset $K$ of $G$, $AK\neq G$ and $(G\setminus A)K\neq G$. Fix any $A\in\mathscr{A}$. We claim that $A$ and $G\setminus A$ are dense in $G$. Indeed, if $A$ is not dense in $G$, then $Int(G\setminus A)\neq\emptyset$; since $G$ is precompact, there exists a finite subset $K$ of $G$ such that $K (G\setminus A)=G$, which is a contradiction. If $G\setminus A$ is not dense in $G$, then we can obtain a contradiction by a similar method. Therefore, $A$ and $G\setminus A$ are dense in $G$.
\end{proof}

\begin{corollary}
Each infinite precompact pre-topological group is resolvable; in particular, each infinite precompact topological group is also resolvable.
\end{corollary}

Indeed, by \cite[Corollary 10]{MP1996}, we have the following more general theorem by a similar proof of Theorem~\ref{t202224}.

\begin{theorem}
Each uncountable $\omega$-narrow pre-topological group is resolvable; in particular, each uncountable $\omega$-narrow topological group is also resolvable.
\end{theorem}

\end{document}